\documentclass[12pt]{amsart}
\usepackage{amsmath,amsfonts,amssymb,amsthm,amstext,pgf,graphicx,hyperref,verbatim,mathrsfs,ytableau,textcomp,color,young,youngtab,tikz,tikz-cd,bbold,cancel,enumitem}
\usepackage{cleveref}
\definecolor{asparagus}{rgb}{0.0, 0.5, 0.0}
\makeatletter

\makeatletter
\newcommand\xlabel[2][]{\phantomsection\def\@currentlabelname{#1}\label{#2}}
\makeatother
\theoremstyle{plain}
\newtheorem{thm}{Theorem}[section]
\newtheorem{lem}[thm]{Lemma}

\newtheorem{cor}[thm]{Corollary}

\theoremstyle{definition}
\newtheorem{defn}{Definition}[section]

\newtheorem{rem}{Remark}[section]
\newtheorem{opqn}{Open Question}
\DeclareMathOperator{\1}{id}
\DeclareMathOperator{\e}{e}
\DeclareMathOperator{\T}{\mathbf{Term}}
\DeclareMathOperator{\Tal}{\mathbf{Tally}}
\DeclareMathOperator{\altch}{Ch}
\DeclareMathOperator{\trivial}{triv}
\DeclareMathOperator{\Tr}{trace}
\DeclareMathOperator{\Ustd}{UStd(\lambda)}

\DeclareMathOperator{\MT}{\mathcal{S}^{MT}}
\DeclareMathOperator{\ET}{\mathcal{S}^{ET}}

\DeclareMathOperator{\Ustdc}{UStd(\lambda^{\prime})}
\DeclareMathOperator{\std}{Std(\lambda)}

\DeclareMathOperator{\Par}{Par(n)}
\DeclareMathOperator{\Cpar}{CPar(n)}
\DeclareMathOperator{\Ncpar}{NCPar(n)}

\title{}
\begin{document}
	\title[Limit profile for the TT$2$R shuffle]{Limit profile for the transpose top-$2$ with random shuffle}
	\author{Subhajit Ghosh}
	\address[Subhajit Ghosh]{Department of Mathematics, Indian Institute of Technology Madras, Chennai 600 036}
	\email{gsubhajit@alum.iisc.ac.in}
	\author{Nishu Kumari}
	\address[Nishu Kumari]{Department of Mathematics, Institute of Mathematical Sciences, Chennai 600113}
	\email{nishukumari@alum.iisc.ac.in}
	\makeatletter
	\@namedef{subjclassname@2020}{%
		\textup{2020} Mathematics Subject Classification}
	\makeatother
	%----------Abstract--------------
	\begin{abstract}
		The transpose top-$2$ with random shuffle (J. Theoret. Probab., 2020) is a lazy random walk on the alternating group $A_n$ generated by $3$-cycles of the form $(\star,n-1,n)$ and $(\star,n,n-1)$. We obtain the limit profile of this random walk by comparing it with the random walk on $A_n$ generated by all $3$-cycles. Our method employs a non-commutative Fourier analysis analogue of the comparison method introduced by Nestoridi (Electron. J. Probab., 2024). We also give the complete spectrum of alternating group graph, thus answering a question of Huang and Huang (J. Algebraic Combin., 2019).
	\end{abstract}
	\maketitle
	\section{Introduction}\label{sec:intro}
    The transpose top-$2$ with random shuffle, introduced by the first named author in 2020 \cite{TT2R}, is a lazy simple random walk on the Cayley graph $\widetilde{AG}_n$ of the alternating group $A_n$ with the generating set $\{(i,n-1,n),(i,n,n-1):1\leq i\leq n-2\}$. We may assume $\widetilde{AG}_n$ as the \emph{alternating group graph} $AG_n$; because, they are isomorphic (one is obtained from the other by relabelling the vertices). Formally, $AG_n$ is the Cayley graph of the alternating group $A_n$ with the generating set $\{(1,i,2),(1,2,i):3\leq i\leq n\}$. Jwo et al. introduced the alternating group graph in 1993 \cite{TT2R_Cayley_graph0}, afterward the alternating group graph caught considerable attention in computer science and mathematics \cite{AGG_4,AGG_1,AGG_2,AGG_0,AGG_3,AGG_5}. In general, Cayley graphs provide a very natural and a rich framework for the design and analysis of interconnection networks for parallel computers \cite{CG-appl}. 
    
	As a random walk on the alternating group $A_n$, the transpose top-$2$ with random shuffle is driven by the following probability measure (defined on $A_n$):
	\begin{equation}\label{eq:TT2R_defn}
		P(\pi)=
		\begin{cases}
			\frac{1}{2n-3}&\text{ if }\pi\in \{(i,n-1,n),(i,n,n-1):1\leq i\leq n-2\},\\%=(i,n,n-1)\text{ for }1\leq i\leq n-2,\\
			%\frac{1}{2n-3}&\text{ if }\pi=(i,n-1,n)\text{ for }1\leq i\leq n-2,\\
			\frac{1}{2n-3}&\text{ if }\pi=\1,\text{ the identity permutation},\\
			0&\text{ otherwise}.
		\end{cases}
	\end{equation}
    The name `transpose top-$2$ with random shuffle' was given based on the shuffling algorithm it represents. The random walk model is a lazy variant of the following process: \emph{First, the top two cards are transposed. Then, one of these two cards is selected with equal probability, and it is swapped with a randomly chosen card from the remaining $n-2$ cards.} The transpose top-$2$ with random shuffle is irreducible and aperiodic \cite[Proposition 1.2]{TT2R}; thus, the distribution after $k$ transitions converges to the stationary distribution as $k$ goes to infinity. The stationary distribution in this case is the uniform distribution $U_{A_n}$ defined on the alternating group $A_n$ (because this is a simple random walk on a finite connected regular graph). The first named author has shown that the transpose top-$2$ with random shuffle exhibits \emph{total variation cutoff phenomenon} at time $n\log n$ with a \emph{cutoff window} of order $n$ \cite[Theorem 1.1, Theorem 4.7]{TT2R}. However, there is no issue in considering $n\log n+o(n\log n)$ as a cutoff time, thanks to the definition of the cutoff phenomenon. Throughout this article, we take the cutoff time for the transpose top-$2$ with random shuffle to be $\left(n-\frac{3}{2}\right)\log n$ for simplicity in calculations.

    We now outline an intuitive argument to show that $\left(n-\frac{3}{2}\right)\log n+O(n)$ steps are expected. To illustrate this, consider the following algorithmic process for marking the cards:
    \begin{enumerate}[left=0em, itemsep=0em]
        \item Start by marking one card in the deck.
        \item \label{S-I} Select an element from $\{(i,n,n-1),(i,n-1,n),\1:1\leq i\leq n-2\}$
        uniformly at random.
        \item \label{S-II} If the element selected in \hyperref[S-I]{(2)} is
        \begin{itemize}
            \item $\1$, then mark the $n$th card, and keep the deck unaltered.
            \item $(i,n-1,n)$, then mark the $(n-1)$th card and then perform the following: First, transpose the $n$th and $(n-1)$th cards, and then swap the $i$th and $n$th cards.
            \item $(i,n,n-1)$, then mark the $n$th card and then perform the following: Transpose the $n$th and $(n-1)$th cards, and then swap the $i$th and $(n-1)$th cards.
        \end{itemize}
        \item \label{S-III} If all the cards are marked, stop. Else, return to \hyperref[S-I]{(2)} and continue.
    \end{enumerate}
    We observe that the relative order of the marked cards is close to a random even permutation once the cards at positions  $n$ and $(n-1)$ are marked, given knowledge of which cards are marked and their positions. This can be shown by induction on the number of marked cards in $\{1, \dots, n-2\}$, utilizing the fact that the deck is permuted only by 3-cycles, which are even permutations. Consequently, once all cards are marked, the entire deck undergoes a random even permutation.
    
    Additionally, the probability that a card at position $i\in\{1,\dots,n-2\}$ is marked is $\frac{2}{2n-3}$. Once a card at position $i\in\{1,\dots,n-2\}$ is marked, it remains marked forever. After all the other cards have been marked, it takes at most $O(n)$ additional steps to mark the $n$th and $(n-1)$th cards. Thus, a coupon collector-type argument ensures that all the cards will be marked after approximately $\left(n - \frac{3}{2}\right) \log n + O(n)$ steps.
    
    In this article, we obtain the \emph{limit profile} for the transpose top-$2$ with random shuffle.
	
	The cutoff phenomenon of a finite (irreducible and aperiodic) Markov chain ensures the convergence to stationary distribution occurs suddenly over a very short time, known as the cutoff window. In real-world applications, knowing that a random process exhibits a cutoff can save time and the running costs of associated algorithms. The first mathematical demonstration of the cutoff phenomenon was due to Diaconis and Shahshahani in 1981, who proved it to hold for the random transposition model \cite{DS}. During the following years, Diaconis, Aldous, and their collaborators developed the theory in a long series of papers. Now, it has become one of the vast subfields in probability literature. For a survey on this topic, we refer the reader to the references \cite{AD1,AD2,D1,D2,LPW,S}.
	
	A relatively new direction is that of the limit profile, a function that precisely describes the sharp transition at the cutoff window. The limit profile is known only for a handful number of Markov chains, viz. the random walk on the hypercube \cite{diaconis1990asymptotic} (a short argument was recently provided in \cite[Theorem 5.1]{N-limit}), the riffle shuffle \cite{Riff-sh}, the asymmetric exclusion process on the segment \cite{ASEP-limit}, the simple exclusion process on the cycle \cite{SEP-limit}, the Bernoulli–Laplace process \cite{olesker2024limit}, the projections of random walks on groups \cite{Neato-Oles_Tay}, random walks on the abelian groups \cite{hermon2021cutoff}, the simple random walk on Ramanujan graphs \cite{Ramanujan_graph-limit}, and a few random walks (random transposition \cite{T-limit} and star transposition shuffles \cite{N-star_limit}) on the symmetric group. More recently, Delhaye obtained a profile result for the quantum unitary group \cite{delhaye2024brownian}. Teyssier studied the limit profile for the classical random transposition model \cite{T-limit} (recently, Jain and Sawhney provided an alternative proof of Teyssier's result in \cite{jain2024hitting}). It has a connection with a magnificent phenomenon in the theory of mixing times, which informally says, ``occasionally, certain aspects of a system mix much faster than the system as a whole" \cite{Comp-mix1,Comp-mix2}, and supports a conjecture of Nathana\"{e}l Berestycki \cite[Conjecture 1.2]{T-limit}. Afterward, Nestoridi et al. developed some methods to obtain the limit profile further to reversible Markov chains and applied it to some models \cite{N-star_limit,N-limit,Neato-Oles_Tay}. In this article, we provide a Fourier analysis analogue of Nestoridi's comparison method \cite{N-star_limit}. Our technique compares random walks on a finite group. Importantly, our result does not assume simultaneous diagonalizability of the transition matrices; the comparison result is presented in \Cref{thm:comparison}. The formal definitions of the cutoff phenomenon and the limit profile will be given in \Cref{RWFG}.
    
	Now, we recall the definition of the \emph{total variation distance} between probability measures on a finite set. Let $\mathcal{P}$ and $\mathcal{Q}$ be two probability measures on a finite set $\Omega$. Then the \emph{total variation} distance between $\mathcal{P}$ and $\mathcal{Q}$, denoted $\|\mathcal{P}-\mathcal{Q}\|_{\text{TV}}$, is defined by
	\begin{equation}\label{eq:TV-def}
		\|\mathcal{P}-\mathcal{Q}\|_{\text{TV}}:=\sup_{A\subseteq\Omega}|\mathcal{P}(A)-\mathcal{Q}(A)|=\frac{1}{2}\sum_{\omega\in\Omega}|\mathcal{P}(\omega)-\mathcal{Q}(\omega)|.
	\end{equation}
	We now state the main result of this paper.
	\begin{thm}\label{main-thm}
		Let $c\in\mathbb{R}$, and $d_{\emph{TV}}\left(\emph{Poi}(1+e^{-c}),\emph{Poi}(1)\right)$ denote the total variation distance between the laws of the Poisson distributions with parameters $1+e^{-c}$ and $1$. Then the limit profile for the transpose top-$2$ with random shuffle is given by $d_{\emph{TV}}\left(\emph{Poi}(1+e^{-c}),\emph{Poi}(1)\right)$ for every real number $c$, i.e.,
		\[\lim_{n\rightarrow\infty}\left\|P^{*\lceil\left(n-\frac{3}{2}\right)\log n+cn\rceil}-U_{A_n}\right\|_{\text{TV}}=d_{\emph{TV}}\left(\emph{Poi}(1+e^{-c}),\emph{Poi}(1)\right),\;c\in\mathbb{R}.\]
	\end{thm}
	Let us recall the random walk on $A_n$ generated by all $3$-cycles in $A_n$, i.e., it is the random walk on $A_n$ driven by the probability measure $Q$, defined on $A_n$, as follows:
	\begin{equation}\label{eq:3-cycle_defn}
		Q(\pi)=
		\begin{cases}
			\frac{3}{n(n-1)(n-2)}&\text{ if }\pi\text{ is a }3\text{-cycle in }A_n,\\
			0&\text{ otherwise}.
		\end{cases}
	\end{equation}
	The random walk on $A_n$ generated by all $3$-cycles satisfies the cutoff phenomenon with time $\frac{n}{3}\log n$ and window $O(n)$. The limit profile of this random walk is $d_{\text{TV}}\left(\text{Poi}(1+e^{-c}),\text{Poi}(1)\right)$, thanks to Nestoridi and Olesker-Taylor \cite{N-limit}. We prove \Cref{main-thm} by comparing the transpose top-$2$ with random shuffle and the random walk on $A_n$ generated by all the $3$-cycles. We conclude this section by giving the organisation of this article.
	
	\subsection*{Organisation of this paper} In \Cref{RWFG}, we focus on the random walks on a finite group and provide our comparison method, which relies on the Fourier analysis of the group. In \Cref{sec:spectrum}, we recall the necessary representation theory of $A_n$ and lay the ground work for proving  \Cref{main-thm}. Finally, we prove \Cref{main-thm} in \Cref{sec:main-proof}. On a purely graph-theoretic note, we will answer a question asked by Huang and Huang \cite[Problem 1]{AGG_1} in \Cref{sec:graph-spectrum}. Finally, in \Cref{sec:example}, we present an example illustrating our comparison method for non-commuting transition matrices.
%%%%%%%%%%%%%%%%%%%%%%%%%%%%%%%%%%%%%%%%%%%%%%%%
%%%%%%%%%%%%%%%%%%%%%%%%%%%%%%%%%%%%%%%%%%%%%%%%
\section{Comparison of limit profiles for random walks on a finite group}\label{RWFG}
	In this section, we give a method for comparing the limit profiles for various random walks on a finite group. We briefly recall the random walks on a finite group and the representation theory of the group. Then, the main result of this section will be proved. We end this section with some remarks on our comparison technique.
	
	Let $G$ be a finite group, and $\gamma$ be a probability measure on $G$.  Then, the  (left-invariant) \emph{random walk on $G$ driven by $\gamma$} is a time homogeneous (discrete-time) Markov chain $\{X_t\}_{t=0}^{\infty}$ with state space $G$ and one-step transition probabilities
	\[\mathbb{P}\left(X_1=y\mid X_0=x\right):=\gamma(x^{-1}y),\text{ for all }x,y\in G.\]
	Fix an initial distribution $\gamma_0$. Let $\{Y_0,Y_1,Y_2,Y_3,\dots\}$ be a sequence of independent $G$-valued random variables such that $Y_0$ has law $\gamma_0$ and $Y_1,Y_2,Y_3,\dots$ have identical law $\gamma$. Then, the left-invariant random walk defined above can be obtained as
	\begin{equation}\label{eq:rwfg-def}
		X_k:=Y_0Y_1Y_2\dots Y_k\text{ for all }k\geq 1.
	\end{equation}
	Given the initial law $\gamma_0$, the distribution after $k$ transitions is given by the law of $X_k$. The law of $X_k$ is given by $\gamma_0*\gamma^{*k}$, where $\gamma^{*k}$ is the $k$-fold self-convolution of $\gamma$. Recall that the convolution of two real valued functions $\alpha$ and $\beta$ (defined on $G$), denoted $\alpha*\beta$, is defined by
	\[\alpha*\beta(x):=\sum_{g\in G}\alpha(g)\beta(g^{-1}x)\text{ for all }x\in G.\]
	The uniform measure $U_G$ given by $U_G(g)=\frac{1}{|G|}$ satisfies $U_G=U_G*\gamma$; thus, it is a stationary distribution of the random walk on $G$ driven by $\gamma$. Stationary distribution is unique when the random walk is irreducible. \emph{The random walk on $G$ driven by the probability measure $\gamma$ is irreducible if and only if the support of $\gamma$,  i.e., the set $\{x\in G:\gamma(x)>0\}$, generates the group $G$} \cite[Proposition 2.3]{S}. Moreover, if the random walk is aperiodic then the law of $X_k$ converges to the stationary distribution $U_G$ as $k\rightarrow\infty$. The random walk $\{X_t\}_{t=0}^{\infty}$ is reversible if and only if $\gamma(g)=\gamma(g^{-1})$ for all $g\in G$. %(i.e., $\gamma$ is symmetric). 
	For any $x\in G$, let $\delta_x$ be the probability measure on $G$ that takes value $1$ at $x$ and $0$ elsewhere. Then, we have the following:
	\[\|\delta_x*\gamma^{*k}-U_G\|_{\text{TV}}=\|\delta_y*\gamma^{*k}-U_G\|_{\text{TV}}\text{ for all }x,y\in G.\]
	For an irreducible and aperiodic random walk on the group $G$ driven by the probability measure $\gamma$, the (total variation) \emph{mixing time} is a measure of the number of transitions required for the random walk to approach $U_G$ up to a given tolerance. More formally, given $\varepsilon>0$, the $\varepsilon$-mixing time, denoted $t_{\text{mix}}(\varepsilon)$, is defined by
	\[t_{\text{mix}}(\varepsilon):=\min\{k:\|\gamma^{*k}-U_G\|_{\text{TV}}<\varepsilon\}.\]
    
    Now, we are in a position to define the cutoff phenomenon and limit profile; both of these concepts are defined for a sequence of random walks. 
	\begin{defn}\label{def:cutoff-limit_profile}
		Let $\{G_n\}_n$ be a sequence of finite groups. For each $n\geq1$, let $\gamma_n$ be a probability measure defined on $G_n$ such that the random walk on $G_n$ driven by $\gamma_n$ is irreducible and aperiodic. The sequence is said to satisfy the \emph{total variation cutoff phenomenon} if there are sequences $\{\tau_n\}_n$ (\emph{cutoff time}) and  $\{w_n\}_n$ (\emph{cutoff window}) such that $\tau_n\rightarrow\infty$, $w_n=o(\tau_n)$, and the following holds:
		\[\lim_{c\rightarrow-\infty}\liminf_{n\rightarrow\infty}\|\gamma_n^{*\lceil\tau_n+cw_n\rceil}-U_G\|_{\text{TV}}=1,
        \,\,\lim_{c\rightarrow\infty}\limsup_{n\rightarrow\infty}\|\gamma_n^{*\lceil\tau_n+cw_n\rceil}-U_G\|_{\text{TV}}=0.\]
		The (total variation) \emph{limit profile} can be formally defined as a function $f:\mathbb{R}\rightarrow\mathbb{
			R}$ such that
		\begin{equation*}\label{eq:limit_profile-def}
			f(c):=\lim_{n\rightarrow\infty}\|\gamma_n^{*\lceil\tau_n+cw_n\rceil}-U_G\|_{\text{TV}},
		\end{equation*}
		provided the limit exists for each (fixed) real number $c$. In case, the limit does not exist, similar definition could be given for the $\limsup$ and $\liminf$.
	\end{defn}
    \begin{rem}
        Given a sequence of irreducible and aperiodic Markov chains, if we denote the $\varepsilon$-mixing time of the $n$th chain by $t^{(n)}_{\text{mix}}(\varepsilon)$, then the usual definition of the cutoff phenomenon says $t^{(n)}_{\text{mix}}(\varepsilon)\rightarrow\infty$ and $\displaystyle\lim_{n\rightarrow\infty} t^{(n)}_{\text{mix}}(1-\varepsilon)/t^{(n)}_{\text{mix}}(\varepsilon)=1$ for all $0<\varepsilon<1$.
        \Cref{def:cutoff-limit_profile} presents an equivalent definition of the cutoff phenomenon.
    \end{rem}
	We now focus on the representation theory of finite groups. Let $V$ be a finite-dimensional complex vector space and GL$(V)$ be the group of all invertible linear operators on $V$. Let $G$ be a finite group. Let $I$ denote the identity element of GL$(V)$ (i.e. the identity operator on $V$) and $\e$ denote the identity element of $G$. A (complex) \emph{linear representation} $(\rho,V)$ of $G$ is a homomorphism $\rho:G\rightarrow \text{GL}(V)$. In particular, $\rho(\e)=I$ and $\rho(g^{-1})=\rho(g)^{-1},\;g\in G$. The dimension of the vector space $V$ is said to be the \emph{dimension} of the representation $\rho$ and is denoted by $d_{\rho}$. The representation space $V$ is called the \emph{$G$-module} corresponding to the representation $\rho$. Given $\rho$, we simply say $V$ is a representation of $G$. For example, let $V$ be one-dimensional. Then, $\trivial:G\rightarrow \text{GL}(V)$, defined by $\trivial(g)\mapsto(v\mapsto v)$, for all $v\in V$ and $g\in G$, is a representation of $G$, known as the \emph{trivial representation} of $G$. We now define the right regular representation of $G$. 
	\begin{defn}
		Let $\mathbb{C}[G]$ be the group algebra consisting of all formal linear combinations of the elements of $G$ with complex coefficients, i.e. $\mathbb{C}[G]=\{\sum_{g}c_gg\mid c_g\in\mathbb{C},\; g\in G\}$. Then the \emph{right regular representation} $R:G\longrightarrow \text{GL}(\mathbb{C}[G])$ of $G$ is defined by 
		\[R(g)\left(\displaystyle
		\sum_{h\in G}C_hh\right)=\displaystyle\sum_{h\in G}C_hhg^{-1},\quad C_h\in\mathbb{C},\]
		i.e., $R(g)$ is an invertible matrix over $\mathbb{C}$ of order $|G|\times|G|$.
	\end{defn}
	For $g\in G$, the trace of the matrix $\rho(g)$ is said to be the \emph{character} value of $\rho$ at $g$ and is denoted by $\chi^{\rho}(g)$. The character values are constants on conjugacy classes, i.e., the characters are class functions. We also have $\chi^{\rho}(\e)=d_{\rho}$, and $\chi^{\rho}(g^{-1})=\overline{\chi^{\rho}(g)}$, the complex conjugate of $\chi^{\rho}(g)$. A vector subspace $W$ of $V$ is said to be \emph{stable} (or \emph{invariant}) under $\rho$ if $\rho(g)\left(W\right)\subset W$ for all $g$ in $G$. \emph{If $W$ is a stable subspace of $V$ under $\rho$, then there exists a complement $W^0$ of $W$ in $V$ which is stable under $\rho$} (\cite[Theorem 1]{Serre}). The representation $\rho$ is \emph{irreducible} if $V$ has no non-trivial proper stable subspace. For example the trivial representation defined above is irreducible. Two representations $(\rho_1,V_1)$ and $(\rho_2,V_2)$ of $G$ are are said to be \emph{isomorphic} if there exists an invertible linear map $T:V_1\rightarrow V_2$ such that $T\circ\rho_1(g)=\rho_2(g)\circ T$ for all $g\in G$. Schur's lemma says that \emph{If a group algebra element $\mathfrak{g}\in\mathbb{C}[G]$ commutes with every element of the group $G$, then $\mathfrak{g}$ acts as a scalar on the irreducible $G$-modules} \cite[Proposition 5]{Serre}. We denote the set of all (non-isomorphic) irreducible representations of $G$ using notation $\widehat{G}$. \emph{The right regular representation of $G$ decomposes into irreducible representations with multiplicity equal to their respective dimensions} \cite[p. 18, Corollary 1]{Serre}. Thus we have the following:
	\begin{equation}\label{eq:Group_alg._decom.}
		\mathbb{C}[G]\cong\underset{\rho\in\widehat{G}}{\oplus}\;d_{\rho}V^{\rho},
	\end{equation}
	where $V^{\rho}$ is the irreducible $G$-module corresponding to $\rho\in\widehat{G}$ with dimension $d_{\rho}$. We also have $\displaystyle\sum_{\rho\in\widehat{G}}d_{\rho}^2=|G|$ by equating the dimensions in \eqref{eq:Group_alg._decom.}.
	
	We now define the Fourier transform of a real valued function on $G$. Let $\phi:G\rightarrow\mathbb{R}$ be a function and $(\rho,V)$ be a representation of $G$. Then, the \emph{Fourier transform} of $\phi$ at $\rho$, denoted $\widehat{\phi}(\rho)$, is defined as an operator on $V$ given by 
	\[\widehat{\phi}(\rho):=\sum_{g\in G}\phi(g)\rho(g).
    % ,\quad\rho\text{ is a representation of }G.
    \]
	Given two functions $\phi,\psi:G\rightarrow \mathbb{R}$, we have $\widehat{\phi*\psi}(\rho)=\widehat{\phi}(\rho)\circ\widehat{\psi}(\rho)$, here $\circ$ denotes the composition of operators. If an ordered basis of $V$ is understood from the context, then we simply think $\widehat{\phi}(\rho)$ and $\widehat{\psi}(\rho)$ as matrices with respect to the basis. In that case $\circ$ is the matrix multiplication. We now recall the Plancherel formula \cite[Theorem 4.1]{D1} below.
	\begin{equation}\label{eq:Plancherel formula}
		\sum_{x\in G}\phi(x^{-1})\psi(x)=\frac{1}{|G|}\sum_{\rho\in\widehat{G}}d_{\rho}\Tr\left(\widehat{\phi}(\rho)\widehat{\psi}(\rho)\right).
	\end{equation}
	For a random walk on $G$ driven by the probability measure $\gamma$ (defined on the group $G$), the transition matrix is given by $\widehat{\gamma}(R)$. Here $R$ is the right regular representation defined above. Now, we introduce our comparison method; the main result of this section is given below.
	\begin{thm}\label{thm:comparison}
		Let $\{G_n\}_{n=1}^{\infty}$ be a sequence of finite groups. For each $n\geq 1$, let $\nu_n$ and $\mu_n$ be two probability measures defined on $G_n$ such that the random walks on $G_n$ driven by them are irreducible, aperiodic, and reversible. Assume that the random walk on $G_n$ driven by $\nu_n$ satisfies cutoff phenomenon at time $\tau_{\nu,n}$ with window of order $w_{\nu,n}$, and it has limit profile
		\begin{equation}\label{eq:hypo1}
			f(c):=\lim_{n\rightarrow\infty}\left\|\nu_n^{*\lceil\tau_{\nu,n}+cw_{\nu,n}\rceil}-U_{G_n}\right\|_{\emph{TV}},\quad\text{for all }c\in\mathbb{R}.
		\end{equation}
		If there exist real numbers $\tau_{\mu,n}$ and $w_{\mu,n}$ such that $w_{\mu,n}=o\left(\tau_{\mu,n}\right)$ and
		\begin{equation}\label{eq:hypo2}
			\lim_{n\rightarrow\infty}
            \sum_{\rho\in\widehat{G}}
            d_{\rho} 
            \Tr\left(\left(\widehat{\nu_n}(\rho)\right)^{\lceil\tau_{\nu,n}+cw_{\nu,n}\rceil}-\left(\widehat{\mu_n}(\rho)\right)^{\lceil\tau_{\mu,n}+cw_{\mu,n}\rceil}\right)^2=0,
		\end{equation}
		then the random walk on $G_n$ driven by $\mu_n$ exhibits cutoff phenomenon at time $\tau_{\mu,n}$ with a window of order $w_{\mu,n}$; moreover, its limit profile is given by $f(c)$.
	\end{thm}
	\begin{proof}
		We first note that the measures $\nu_n$ and $\mu_n$ are symmetric, i.e., $\nu_n(g)=\nu_n(g^{-1})$ and $\mu_n(g)=\mu_n(g^{-1})$ for all $g\in G_n$, because the random walks on $G_n$ driven by $\mu_n$ and $\nu_n$ are both reversible, and the stationary distribution is the uniform distribution on $G_n$. More precisely,
        \[U_{G_n}(\e)\;\;\times\;\;\text{\begin{minipage}{1.6in}
        one-step transition probability from $\e$ to $g$
        \end{minipage}}\;\;=\;\;U_{G_n}(g)\;\;\times\;\;\text{\begin{minipage}{1.6in}
        one-step transition probability from $g$ to $\e$
        \end{minipage}}\;,\]
        for all $g\in G_n$. Here, $e$ denotes the identity element of $G_n$.
        
        Let $c\in\mathbb{R}$. Throughout the proof we write $\tau_{\nu,n}+cw_{\nu,n}$ as $t_{\nu}$, and $\tau_{\mu,n}+cw_{\mu,n}$ as $t_{\mu}$ to avoid notational complication. Now we have that
		\begin{align}\label{eq:comparison_method1}
			&\Big|\left\|\nu_n^{*\lceil t_{\nu}\rceil}-U_{G_n}\right\|_{\text{TV}}-\left\|\mu_n^{*\lceil t_{\mu}\rceil}-U_{G_n}\right\|_{\text{TV}}\Big|\nonumber\\
			=\;&\Big|\frac{1}{2}\sum_{g\in G_n}\left(\left|\nu_n^{*\lceil t_{\nu}\rceil}(g)-U_{G_n}(g)\right|-\left|\mu_n^{*\lceil t_{\mu}\rceil}(g)-U_{G_n}(g)\right|\right)\Big|,\text{ from definition \eqref{eq:TV-def}}\nonumber\\
			\leq\;&\frac{1}{2}\sum_{g\in G_n}\Big|\left|\nu_n^{*\lceil t_{\nu}\rceil}(g)-U_{G_n}(g)\right|-\left|\mu_n^{*\lceil t_{\mu}\rceil}(g)-U_{G_n}(g)\right|\Big|,\text{ by triangle inequality}\nonumber\\
			\leq\;&\frac{1}{2}\sum_{g\in G_n}\Big|\left(\nu_n^{*\lceil t_{\nu}\rceil}(g)-U_{G_n}(g)\right)-\left(\mu_n^{*\lceil t_{\mu}\rceil}(g)-U_{G_n}(g)\right)\Big|,\text{ using triangle inequality}\nonumber\\
			=\;&\frac{1}{2}\sum_{g\in G_n}\Big|\nu_n^{*\lceil t_{\nu}\rceil}(g)-\mu_n^{*\lceil t_{\mu}\rceil}(g)\Big|\leq\;\frac{1}{2}\;\sqrt{\sum_{g\in G_n}|G_n|\left(\nu_n^{*\lceil t_{\nu}\rceil}(g)-\mu_n^{*\lceil t_{\mu}\rceil}(g)\right)^2},
		\end{align}
	where	the inequality in \eqref{eq:comparison_method1} follows from Cauchy–Schwarz inequality. For every $g\in G_n$, let us set 
		\[\eta_n(g):=\nu_n^{*\lceil t_{\nu}\rceil}(g)-\mu_n^{*\lceil t_{\mu}\rceil}(g).\]
		The self-convolution of a symmetric measure is symmetric; therefore, $\eta_n(g)=\eta_n(g^{-1})$ for all $g\in G$. Thus, the Plancherel formula \eqref{eq:Plancherel formula} and \eqref{eq:comparison_method1} implies
		\begin{align}
			\Big|\left\|\nu_n^{*\lceil t_{\nu}\rceil}-U_{G_n}\right\|_{\text{TV}}-\left\|\mu_n^{*\lceil t_{\mu}\rceil}-U_{G_n}\right\|_{\text{TV}}\Big|^{2}\leq\;&\frac{1}{4}\sum_{\rho\in \widehat{G}_n} d_{\rho} \Tr\left(\widehat{\eta}_n(\rho)\right)^2 \nonumber\\
			=&\frac{1}{4}\sum_{\rho\in \widehat{G}_n}  d_{\rho} \Tr\left(\widehat{\nu_n^{*\lceil t_{\nu}\rceil}}(\rho)-\widehat{\mu_n^{*\lceil t_{\mu}\rceil}}(\rho)\right)^2 \label{eq:comparison_method2}\\	
			=&\frac{1}{4}\sum_{\rho\in \widehat{G}_n}  d_{\rho} \Tr\left(\left(\widehat{\nu_n}(\rho)\right)^{\lceil t_{\nu}\rceil}-\left(\widehat{\mu_n}(\rho)\right)^{\lceil t_{\mu}\rceil}\right)^2.\label{eq:comparison_method3}\
		\end{align}
		The equality in \eqref{eq:comparison_method2} follows from the fact $\widehat{\eta}_n(\rho)=\widehat{\nu_n^{*\lceil t_{\nu}\rceil}}(\rho)-\widehat{\mu_n^{*\lceil t_{\mu}\rceil}}(\rho)$ for all $\rho\in \widehat{G}_n$. Now, letting $n\rightarrow\infty$ in \eqref{eq:comparison_method3}, the hypotheses \eqref{eq:hypo1} and \eqref{eq:hypo2} implies
		\begin{equation}\label{eq:comparison_method}
			\lim_{n\rightarrow\infty}\left\|\mu_n^{*\lceil\tau_{\mu,n}+cw_{\mu,n}\rceil}-U_{G_n}\right\|_{\emph{TV}}=f(c),\quad c\in\mathbb{R}.
		\end{equation}
		The cutoff phenomenon of the random walk on $G_n$ driven by $\nu_n$ ensures
		\[\lim_{c\rightarrow-\infty}f(c)=1\text{ and }\lim_{c\rightarrow\infty}f(c)=0.\]
		Thus, the theorem follows from \eqref{eq:comparison_method}.\qedhere
	\end{proof}
\begin{rem}
    For the case of random walks on a finite group, \Cref{thm:comparison} is the Fourier analysis analogue of Nestoridi's comparison technique \cite[Theorem 5]{N-star_limit}. It is useful when the random walks are defined on a group that only has irreducible representations of `small' dimensions (viz. the dihedral group). It is also useful when the transition matrices are simultaneously block-diagonalizable with blocks of `small' size. In \Cref{sec:example}, we will demonstrate our comparison method (\Cref{thm:comparison}) for the latter case, with each block size at most $2$.
\end{rem}
	%%%%%%%%%%%%%%%%%%%%%%%%%%%%%%%%%%%%%%%%%%%%%%
	%%%%%%%%%%%%%%%%%%%%%%%%%%%%%%%%%%%%%%%%%%%%%%
	\section{The spectrum of the transition matrices}\label{sec:spectrum}
	The main goal of this section is to prepare the platform for the proof of \Cref{main-thm}. We recall the spectrum of the transition matrices for two random walks on $A_n$ driven by $P$ and $Q$. Let us first define some combinatorial objects that will be used for the rest of this paper.
	
	Let $n$ be a positive integer. A \emph{partition} of $n$, denoted $\lambda:=(\lambda_1,\cdots,\lambda_r)\vdash n$, is defined as a weakly decreasing sequence $(\lambda_1,\cdots,\lambda_r)$ of positive integers such that $\sum_{i=1}^{r}\lambda_i=n$. The partition $\lambda$ can be pictorially visualised using its \emph{Young diagram}. The Young diagram of $\lambda$ is a left-justified arrangement of $r$ rows of boxes with $\lambda_i$ boxes in the $i^{\text{th}}$ row. For example there are five partitions of the positive integer $4$ viz. (4), (3,1), (2,2), (2,1,1) and (1,1,1,1), and the corresponding Young diagrams are given in Figure \ref{fig: Yng_diag_with_4_boxes}.
	\begin{figure}[h]
		\centering
		%\normalsize
		%\small
		%\footnotesize
		\tiny{
			$\begin{array}{cclll}
				\yng(4)&\hspace{0.5cm}\yng(3,1)& \hspace{0.5cm}\yng(2,2) & \hspace{0.5cm}\yng(2,1,1) & \hspace{0.75cm}\yng(1,1,1,1)\\
				(4)\;&\;\quad(3,1)&\quad\;(2,2)&\quad(2,1,1)&\;(1,1,1,1)
			\end{array}$}
		\caption{Young diagrams with $4$ boxes.}\label{fig: Yng_diag_with_4_boxes}
	\end{figure}
	The \emph{Young tableaux} of shape $\lambda$ or simply \emph{$\lambda$-tableaux}, are obtained by filling the numbers $1,\dots,n$ in the boxes of the Young diagram of $\lambda$. A $\lambda$-tableau is \emph{standard} if the entries in its boxes increase from left to right along rows and from top to bottom along columns. The set of all standard tableaux of a given shape $\lambda$ is denoted by $\std$. For example, the standard Young tableaux of shape $(3,1)$ are listed in Figure \ref{fig: Stab_of_shape_(3,1)}. We write $d_{\lambda}$ to denote the number of standard Young tableaux of shape $\lambda$.
	\begin{figure}[h]
		\centering
		\tiny{
			$\begin{array}{ccc}
				T_1=\begin{array}{c}\young({{\substack{1}}}{{\substack{2}}}{{\substack{3}}},{{\substack{4}}})\end{array}, &\quad T_2=\begin{array}{c}\young({{\substack{1}}}{{\substack{2}}}{{\substack{4}}},{{\substack{3}}})\end{array}, & \quad T_3=\begin{array}{c}\young({{\substack{1}}}{{\substack{3}}}{{\substack{4}}},{{\substack{2}}})\end{array}
			\end{array}$}
		\caption{Standard Young tableaux of shape $(3,1)$.}\label{fig: Stab_of_shape_(3,1)}
	\end{figure}
	The \emph{content} of a box in row $u$ and column $v$ of a diagram is the integer $v-u$. Given a tableau $T\in\std$, let $b_T(i)$ denote the box in $T$ containing the integer $i$ and its content is denoted by $c(b_T(i))$ for $1\leq i\leq n$. For example, $c(b_{T_1}(1))=0,c(b_{T_1}(2))=1,c(b_{T_1}(3))=2,c(b_{T_1}(4))=-1$ for the standard Young tableau $T_1$ given in Figure \ref{fig: Stab_of_shape_(3,1)}. The \emph{conjugate} the Young diagram $\lambda$, denoted $\lambda^{\prime}$, is obtained by reflecting $\lambda$ with respect to the diagonal consisting of boxes with content $0$. A diagram $\lambda$ is \emph{self-conjugate} if $\lambda^{\prime}=\lambda$. An \emph{upper standard Young tableau} of shape $\lambda$ is a standard Young tableau $T$ such that $c(b_T(2))=1$. For example, $T_1$ and $T_2$ in Figure \ref{fig: Stab_of_shape_(3,1)} are the upper standard Young tableau of shape $(3,1)$. The collection of all upper standard tableaux of shape $\lambda$ is denoted by $\Ustd$. From now on, we denote the cardinality of a (given) set $S$ by $|S|$. A counting argument (see a related discussion in \cite[Section 2, p. 1838]{TT2R}) gives 
	\begin{equation}\label{eq:Ustd-count}
		\begin{cases}
			&|\Ustd\cup\Ustdc|=|\std|=d_{\lambda},\;\text{ for non-self-conjugate }\lambda\vdash n,
			\text{ and }\\
			&d_{\lambda}^+ =d_{\lambda}^-:= |\Ustd|=\;\frac{1}{2}|\std|=d_{\lambda}/2,\; \text{ for self-conjugate }\lambda\vdash n.
		\end{cases}
	\end{equation}
	Let $\Par$ denote the set of all partitions of $n$. We now define two subsets $\Cpar$ and $\Ncpar$ of $\Par$ as follows:
	\begin{align*}
		\Cpar&=\{\lambda\in\Par\mid\lambda=\lambda^{\prime}\},\;\text{ and }\\
		\Ncpar&=\{\lambda\in\Par|\lambda\neq\lambda^{\prime}\text{ and }\lambda_i>\lambda'_i,
		\text{ here }i\text{ is the smallest index satisfying }\lambda_i\neq \lambda'_i\},
	\end{align*}
	i.e., $\Ncpar$ consists of the `fat' non-self-conjugate partitions of $n$ and $\Cpar$ consists of all self-conjugate partitions of $n$. For example, see Figure \ref{fig:CPar+NCPar} (recall Par$(4)$ from Figure \ref{fig: Yng_diag_with_4_boxes}).
	\begin{figure}[h]
		\centering
		\tiny{
			$\begin{array}{cc}
				\text{CPar}(4)=\Bigg\{\begin{array}{c}\yng(2,2)\end{array}\Bigg\},&\text{NCPar}(4)=\Bigg\{\begin{array}{c}\yng(4)\end{array}\hspace*{-0.75ex},\begin{array}{c}\yng(3,1)\end{array}\Bigg\},
			\end{array}$}
		\caption{Example of $\text{CPar}(4)$ and $\text{NCPar}(4)$.}\label{fig:CPar+NCPar}
	\end{figure}
	
	We now briefly recall the representation theory of $A_n$, for more details we refer the the book of   James and Kerber \cite[Chapter 2, \S 2.5]{JK}. For every non-self-conjugate partition $\lambda$ of $n$, there is an irreducible representation of $A_n$; we denote the corresponding irreducible $A_n$-module by $D_{\lambda}$. The dimension of $D_{\lambda}$ is $d_{\lambda}$; moreover, $D_{\lambda}$ and $D_{\lambda'}$ are isomorphic for all non-self-conjugate $\lambda\vdash n$. For each self-conjugate partition $\lambda$ of $n$, there are two non-isomorphic irreducible representations $D^+_{\lambda}$ and $D^-_{\lambda}$ of $A_n$. The dimension of $D^+_{\lambda}$ (respectively, $D^-_{\lambda}$) is $d_{\lambda}^+ $ (respectively, $d_{\lambda}^-$) for every $\lambda\in\Cpar$. Therefore, the set of all irreducible $A_n$-module is given by
	\begin{equation}\label{eq:irr-A_n-module}
		\big\{D_{\lambda}:\lambda\in\Ncpar\big\}\bigcup\big\{D^+_{\lambda},D^-_{\lambda}:\lambda\in\Cpar\big\}.
	\end{equation}
	
	Now, recall $P$ from \eqref{eq:TT2R_defn}. The eigenvalues of $\widehat{P}(\lambda),\lambda\in\Ncpar$ can be obtained from \cite[Theorem 2.4]{TT2R}, and eigenvalues of $\widehat{P}(\lambda^{\pm}),\lambda\in\Cpar$ can be obtained from \cite[Theorem 2.5]{TT2R}. Here, we write $\lambda^+$ (respectively, $\lambda^-$) to denote the index for the irreducible representation $D^+_{\lambda}$ (respectively, $D^-_{\lambda}$). More formally, we have the following.
	\begin{lem}\label{thm:TT2R_eigenvalues_irr}
		For a non-self-conjugate $\lambda\vdash n$, the eigenvalues of $\widehat{P}(\lambda)$ are indexed by the set $\Ustd\cup\Ustdc$. For  self-conjugate $\lambda\vdash n$, $\widehat{P}(\lambda^+)$ and $\widehat{P}(\lambda^-)$ have the same spectrum, and the eigenvalues of $\widehat{P}(\lambda^{\pm})$ are indexed by the set $\Ustd$\footnote{Note that $\Ustd=\Ustd\cup\Ustdc$ for all self-conjugate $\lambda\vdash n$.}. Let $\lambda\vdash n$; suppose $\mathcal{E}_T$ denote the eigenvalue indexed by $T\in \Ustd\cup\Ustdc$. 
		\begin{itemize}
			\item If $n-1$  and $n$ appear in the same row of $T$, then $\mathcal{E}_T=\frac{c(b_T(n))+c(b_T(n-1))}{2n-3}$.
			\item If $n-1$  and $n$ appear in the same column of $T$, then $\mathcal{E}_T=-\frac{c(b_T(n))+c(b_T(n-1))}{2n-3}$.
			\item If $n-1$  and $n$ appear neither in the same row nor in the same row of $T$, then
			\[\begin{cases}
				\mathcal{E}_T=\frac{c(b_T(n))+c(b_T(n-1))}{2n-3},\text{ and }\vspace*{1ex}\\
				\mathcal{E}_S=-\frac{c(b_T(n))+c(b_T(n-1))}{2n-3}.
			\end{cases}\] 
			Here, $S$ is the upper standard Young tableau obtained from $T$ by interchanging the positions of $n$ and $n-1$.
		\end{itemize}
	\end{lem}
	
Recall $Q$ from \eqref{eq:3-cycle_defn}. We now obtain $\widehat{Q}(\lambda)$ for non-self-conjugate $\lambda\vdash n$, and $\widehat{Q}(\lambda^{\pm})$ for self-conjugate $\lambda\vdash n$. The proof is straightforward application of Schur's lemma \cite[Proposition 5]{Serre}, and it is well known in the literature (for instance, see \cite{Bob_Hough}). However, we present the following lemma to make this article self contained.
	\begin{lem}\label{thm:3-cycle_eigenvalues_irr}
		Let $n>4,\;\lambda\vdash n$, and $\chi^{\lambda}$ denote the irreducible character of the symmetric group $S_n$ indexed by $\lambda$. Then, we have the following:
		\[\begin{cases}
			\widehat{Q}(\lambda)=\frac{\chi^{\lambda}((1,2,3))}{d_{\lambda}}\;I_{d_{\lambda}}&\text{ if }\lambda\text{ is a non-self-conjugate partition of }$n$,\vspace*{1ex}\\
			\widehat{Q}(\lambda^+)=\widehat{Q}(\lambda^-)=\frac{\chi^{\lambda}((1,2,3))}{d_{\lambda}}\;I_{(1/2)d_{\lambda}}&\text{ if }\lambda\text{ is a self-conjugate partition of }$n$.
		\end{cases}\]
		Here, $I_k$ denotes the identity matrix of size $k\times k$.
	\end{lem}
\begin{proof}
	Let $D$ be an irreducible $A_n$-module and `$\altch$' be the corresponding irreducible character. Then, $\widehat{Q}(D)$ is the action of the group algebra element $\frac{3}{n(n-1)(n-2)}Q_n$ on $D$, where
	\[Q_n:=\sum_{\substack{\pi\text{ is a}\\3\text{-cycle}}}\pi=\sum_{1\leq i<j<k\leq n}\left((i,j,k)+(i,k,j)\right)\;\;\in\mathbb{C}[A_n].\]
	Let us observe that $Q_n$ commutes with all the elements of $A_n$. Thus, Schur's lemma \cite[Proposition 5]{Serre} implies that $Q_n$ acts on $D$ like scalars. Therefore, we have
	$\widehat{Q}(D)=C_D\;I_{\dim(D)}$ for some constant $C_D\in\mathbb{C}$. Now, applying the $\Tr$ from both side we get that
	\[C_D=\frac{3}{n(n-1)(n-2)}\frac{\binom{n}{3}\left(\altch(1,2,3)+\altch(1,3,2)\right)}{\dim(D)}=\frac{\altch(1,2,3)+\altch(1,3,2)}{2\times\dim(D)}.\]
	Thus, we have that
	\begin{equation}\label{eq:eig.v._3-cyc_rw}
		\widehat{Q}(D)=\frac{\altch(1,2,3)+\altch(1,3,2)}{2\times\dim(D)}\;I_{\dim(D)}.
	\end{equation}
	We first consider a non-self-conjugate partition $\lambda\vdash n$. Recall the irreducible $A_n$-module $D_{\lambda}$ and $\dim(D_{\lambda})=d_{\lambda}$. Let $\altch^{\lambda}$ denote the irreducible character of $A_n$ indexed by $\lambda$. Then, using $\altch^{\lambda}=\chi^{\lambda}$ \cite[Theorem 2.5.7]{JK} and \eqref{eq:eig.v._3-cyc_rw} we have that
	\[\widehat{Q}(\lambda)=\widehat{Q}(D_{\lambda})=\frac{\chi^{\lambda}((1,2,3))}{d_{\lambda}}\;I_{d_{\lambda}}.\]
	
	We now focus on a self-conjugate partition $\lambda\vdash n$. Recall the irreducible $A_n$-modules $D_{\lambda}^{\pm}$ and $\dim(D_{\lambda}^{\pm})=d_{\lambda}^{\pm}=\frac{1}{2}d_{\lambda}$, and denote their (respective) characters by $\altch_{\pm}^{\lambda}$. The cycle type for the $3$-cycle is $(3,1^{n-3})\vdash n$; moreover, $(3,1^{n-3})$ is not the partition formed by the \emph{hook-lengths} of the diagonal boxes in $\lambda$ and $n>4$ (because $\lambda$ is self-conjugate). Therefore, we have $\altch_{\pm}^{\lambda}(1,2,3)=\altch_{\pm}^{\lambda}(1,3,2)=\frac{1}{2}\chi^{\lambda}((1,2,3))$ \cite[Theorem 2.5.7]{JK} and hence \eqref{eq:eig.v._3-cyc_rw} implies
	\[\widehat{Q}(\lambda^{\pm})=\widehat{Q}(D_{\lambda}^{\pm})=\frac{\chi^{\lambda}((1,2,3))}{d_{\lambda}}\;I_{d_{\lambda}}.\qedhere\]
\end{proof}
%%%%%%%%%%%%%%%%%%%%%%%%%%%%%%%%%%%%%%%%%%%%%%
%%%%%%%%%%%%%%%%%%%%%%%%%%%%%%%%%%%%%%%%%%%%%%
\section{Proof of \Cref{main-thm}}\label{sec:main-proof}
	In this section, we prove \Cref{main-thm}. We use \Cref{thm:comparison} and compare the transpose top-$2$ with random shuffle with the random walks generated by all $3$-cycles. Before coming into the proof of \Cref{main-thm}, we will prove some useful results. Given any partition $\lambda$ of $n$, throughout this section we denote the normalised character $\frac{\chi^{\lambda}(1,2,3)}{d_{\lambda}}$ by $\mathcal{C}_{\lambda}$, and the largest part of $\lambda$ by $\lambda_1$.
	\begin{lem}\label{prop:Error-term}
		Given any $\varepsilon>0$ and $c\in\mathbb{R}$, there exists $M=M(c,\varepsilon)>0$ and sufficiently large $N=N(c,\varepsilon,M)>M$ such that
		\[\sum_{\substack{\lambda\in\Ncpar\cup\Cpar\\\lambda_1\leq n-M}}d_{\lambda}\sum_{T\in\Ustd\cup\Ustdc}\left(\mathcal{C}_{\lambda}^{\lceil\frac{n}{3}(\log n+c)\rceil}-\mathcal{E}_T^{\lceil\left(n-\frac{3}{2}\right)(\log n+c)\rceil}\right)^2<\frac{\varepsilon}{2},\]
		for all $n\geq N$.
		Here, the notations have the same meaning as they are given in \Cref{sec:spectrum}. %$($for instance, see Theorem \ref{thm:TT2R_eigenvalues_irr} and Theorem \ref{thm:3-cycle_eigenvalues_irr}$)$.
	\end{lem}
	\begin{proof}
		To prove the lemma, it is enough to show the existence of $M=M(c,\varepsilon)>0$ and sufficiently large $N=N(c,\varepsilon,M)>M$ such that
		\begin{equation}\label{eq:Error-term1}
			\sum_{\substack{\lambda\in\Ncpar\cup\Cpar\\\lambda_1\leq n-M}}d_{\lambda}\sum_{T\in\Ustd\cup\Ustdc}\left(\mathcal{C}_{\lambda}^{\frac{2n}{3}(\log n+c)}+\mathcal{E}_T^{\left(2n-3\right)(\log n+c)}\right)<\frac{\varepsilon}{4},
		\end{equation}
		for all $n\geq N$; because of the following:
		\begin{itemize}
			\item Given any two real numbers $a$ and $b$, we have $(a-b)^2\leq 2(a^2+b^2)$,
			\item $x\leq\lceil x\rceil$ for any real number $x$, and $0\leq \mathcal{C}_{\lambda}^2,\mathcal{E}_T^2\leq 1$.
		\end{itemize}
		
		Let us choose a large enough positive integer $M_1=M_1(c,\varepsilon)$ such that %$\displaystyle\sum_{m\geq M_1}\frac{e^{-2mc}}{m!}\leq\frac{\varepsilon}{8}$. Therefore,
		\begin{equation}\label{eq:Error-term2}
			\sum_{m\geq r}\frac{e^{-2mc}}{m!}\leq\frac{\varepsilon}{8}\text{ for all }r\geq M_1.
		\end{equation}
		For every $\lambda\in\Ncpar\cup\Cpar$, we have $c(b_T(n))+c(b_T(n-1))\leq 2\lambda_1-3$, i.e., by \Cref{thm:TT2R_eigenvalues_irr}, $\mathcal{E}_T^2\leq\left(\frac{2\lambda_1-3}{2n-3}\right)^2$ for all $T\in\Ustd\cup\Ustdc$. Now, for every $r\in\{1,\dots,n-1\}$,
		\begin{align}
			&\sum_{\substack{\lambda\in\Ncpar\cup\Cpar\\\lambda_1\leq n-r}}d_{\lambda}\sum_{T\in\Ustd\cup\Ustdc}\mathcal{E}_T^{2\left(n-\frac{3}{2}\right)(\log n+c)} \nonumber\\
			\leq&\sum_{\substack{\lambda\in\Ncpar\cup\Cpar\\\lambda_1\leq n-r}}d_{\lambda}\sum_{T\in\Ustd\cup\Ustdc}\left(1-\frac{2(n-\lambda_1)}{2n-3}\right)^{2\left(n-\frac{3}{2}\right)(\log n+c)}.\label{eq:Error-term3}
		\end{align}
		Using $1-x\leq e^{-x}$ for all $x\geq 0$ and
		$\displaystyle\sum_{\substack{\lambda\vdash n\;:\;\lambda_1\leq n-r\\\lambda\notin\Ncpar\cup\Cpar}}d_{\lambda}\displaystyle\sum_{T\in\Ustd\cup\Ustdc}e^{-2\left(n-\lambda_1\right)(\log n+c)}\geq 0,$
		the expression in the right hand side of \eqref{eq:Error-term3} is less than or equal to
		\begin{equation}\label{eq:Error-term4}
			\sum_{\lambda\vdash n\;:\;\lambda_1\leq n-r}d_{\lambda}\sum_{T\in\Ustd\cup\Ustdc}e^{-2\left(n-\lambda_1\right)(\log n+c)}.
		\end{equation}
		Recalling $|\Ustd\cup\Ustdc|=d_{\lambda}\text{ or }\frac{d_{\lambda}}{2}$ from \eqref{eq:Ustd-count}, the expression in \eqref{eq:Error-term4} is less than
		\begin{align}
			\displaystyle\sum_{\lambda\vdash n\;:\;\lambda_1\leq n-r}d_{\lambda}^2e^{-2\left(n-\lambda_1\right)(\log n+c)}\leq&\sum_{\lambda_1=1}^{n-r}\sum_{\substack{\xi\vdash (n-\lambda_1)\\\xi_1\leq\lambda_1}}\binom{n}{\lambda_1}^2d_{\xi}^2e^{-2\left(n-\lambda_1\right)(\log n+c)}\label{eq:Error-term5}\\
			\leq&\sum_{\lambda_1=1}^{n-r}\binom{n}{\lambda_1}^2e^{-2\left(n-\lambda_1\right)(\log n+c)}\sum_{\xi\vdash (n-\lambda_1)}d_{\xi}^2.\label{eq:Error-term6}
		\end{align}
		The inequality in \eqref{eq:Error-term5} follows from the fact $d_{\lambda}\leq\binom{n}{\lambda_1}d_{\xi}$ for all $\xi\vdash (n-\lambda_1)$ with $\xi_1\leq\lambda_1$. Since $\displaystyle\sum_{\xi\, \vdash \,(n-\lambda_1)}d_{\xi}^2=(n-\lambda_1)!$, writing $m$ for $n-\lambda_1$, the expression in the right hand side of \eqref{eq:Error-term6} is equal to
		\begin{equation}\label{eq:Error-term7}
			\sum_{m=r}^{n-1}\binom{n}{m}^2e^{-2m(\log n+c)}m!\leq \sum_{m=r}^{n-1}\frac{n^{2m}}{m!}e^{-2m(\log n+c)}=\sum_{m=r}^{n-1}\frac{e^{-2mc}}{m!}<\sum_{m\geq r}\frac{e^{-2mc}}{m!}.
		\end{equation}
		The leftmost inequality in \eqref{eq:Error-term7} follows from $\binom{n}{m}\leq\frac{n^m}{m!}$. Thus, the expression in \eqref{eq:Error-term4}, and hence the expression in the right hand side of \eqref{eq:Error-term3} is less than $\displaystyle\sum_{m\geq r}\frac{e^{-2mc}}{m!}$. Therefore, using \eqref{eq:Error-term2}, we have the following:
		\begin{equation}\label{eq:Error-term-key-I}
			\sum_{\substack{\lambda\in\Ncpar\cup\Cpar\\\lambda_1\leq n-r}}d_{\lambda}\sum_{T\in\Ustd\cup\Ustdc}\mathcal{E}_T^{2\left(n-\frac{3}{2}\right)(\log n+c)}<\frac{\varepsilon}{8}\text{ for all }r\geq M_1.
		\end{equation}
		
		On the other hand, \cite[Lemma 3.7]{N-limit} ensures the existence of $M=M(c,\varepsilon)\geq M_1$ and sufficiently large $N=N(c,\varepsilon,M)>M$ such that
		\begin{equation}\label{eq:Error-term8}
			\sum_{\lambda\vdash n\;:\;\lambda_1\leq n-M}d_{\lambda}\left|\mathcal{C}_{\lambda}\right|^{\frac{n}{3}(\log n+c)}<\frac{\sqrt{\varepsilon}}{2\sqrt{2}}\quad\quad\text{ for all }n\geq N,
		\end{equation}
		thanks to Nestoridi and Olesker-Taylor for the delicate proof. Now,
		\begin{align}
			\sum_{\substack{\lambda\in\Ncpar\cup\Cpar\\\lambda_1\leq n-M}}d_{\lambda}\sum_{T\in\Ustd\cup\Ustdc}\mathcal{C}_{\lambda}^{\frac{2n}{3}(\log n+c)}<&\sum_{\substack{\lambda\in\Ncpar\cup\Cpar\\\lambda_1\leq n-M}}d_{\lambda}^2\mathcal{C}_{\lambda}^{\frac{2n}{3}(\log n+c)} \label{eq:Error-term9}\\
			\leq&\sum_{\lambda\vdash n\;:\;\lambda_1\leq n-M}d_{\lambda}^2\mathcal{C}_{\lambda}^{\frac{2n}{3}(\log n+c)} \label{eq:Error-term10}\\
			\leq&\left(\sum_{\substack{\lambda\vdash n\\\lambda_1\leq n-M}}d_{\lambda}\left|\mathcal{C}_{\lambda}\right|^{\frac{n}{3}(\log n+c)}\right)^2.\label{eq:Error-term10.A}
		\end{align}
		The inequality in \eqref{eq:Error-term9} follows from the fact that $|\Ustd\cup\Ustdc|=d_{\lambda}$ or $\frac{d_{\lambda}}{2}$. The inequality in \eqref{eq:Error-term10} holds because 
		$\displaystyle\sum_{\substack{\lambda\vdash n\;:\;\lambda_1\leq n-M\\\lambda\notin\Ncpar\cup\Cpar}}d_{\lambda}^2\mathcal{C}_{\lambda}^{\frac{2n}{3}(\log n+c)}\geq 0.$
		Therefore, \eqref{eq:Error-term8} and the expression in \eqref{eq:Error-term10.A} implies that
		\begin{equation}\label{eq:Error-term-key-II}
			\sum_{\substack{\lambda\in\Ncpar\cup\Cpar\\\lambda_1\leq n-M}}d_{\lambda}\sum_{T\in\Ustd\cup\Ustdc}\mathcal{C}_{\lambda}^{\frac{2n}{3}(\log n+c)}<\frac{\varepsilon}{8}\quad\quad\text{ for all }n\geq N.
		\end{equation}
		Finally, \eqref{eq:Error-term1} follows from \eqref{eq:Error-term-key-I} and \eqref{eq:Error-term-key-II}, by setting $r=M\geq M_1$. The proof finishes here.
	\end{proof}
    \begin{rem}
        We emphasize that the constant $M$ in \Cref{prop:Error-term} is sufficiently large and depends on $c$ and $\varepsilon$, but not on $n$. The same applies to the upcoming \Cref{prop:Main-term}, \Cref{lem:Error-term'}, and \Cref{lem:Main-term'}.
    \end{rem}
    
	Let $\lambda$ be a Young diagram with $n$ boxes (i.e., $\lambda\vdash n$). An \emph{inner corner} of $\lambda$ is a box whose removal from $\lambda$ results in a valid Young diagram with $n-1$ boxes. We write $\lambda^{\downarrow}$ to denote the set of all Young diagrams obtained by removing an inner corner from $\lambda$. Then, application of the \emph{hook-length formula} provides 
	\begin{equation}\label{eq:dimension-ineq}
		d_{\zeta}\leq\frac{4^{n-\lambda_1}}{n}d_{\lambda},\;\text{ for all }\zeta\in\lambda^{\downarrow}\text{ satisfying }\zeta_1=\lambda_1\;\; \text{(see \cite[Lemma 15]{N-star_limit})}.
	\end{equation}
	\begin{lem}\label{prop:Main-term}
		Given any $\varepsilon>0$ and $c\in\mathbb{R}$, recall $M$ from \Cref{prop:Error-term}. Then,
		\[\lim_{n\rightarrow\infty}\sum_{\substack{\lambda\in\Ncpar\cup\Cpar\\n-M\leq \lambda_1<n}}d_{\lambda}\sum_{T\in\Ustd\cup\Ustdc}\hspace*{-1ex}\left(\mathcal{C}_{\lambda}^{\lceil\frac{n}{3}(\log n+c)\rceil}\hspace*{-2ex}-\mathcal{E}_T^{\lceil\left(n-\frac{3}{2}\right)(\log n+c)\rceil}\right)^2\hspace*{-1ex}=0.\]
		The notations have the same meaning as they are given in \Cref{thm:TT2R_eigenvalues_irr} and \Cref{thm:3-cycle_eigenvalues_irr}.
	\end{lem}
	\begin{proof}
		For $\lambda\vdash n$, we partition the set $\Ustd\cup\Ustdc$ into two subsets as follows:
		\begin{align*}
			\MT(\lambda)&=\{T\in\Ustd\cup\Ustdc:\text{ both }n-1\text{ and }n\text{ are in the first row of }T\}.\\
			\ET(\lambda)&=\{T\in\Ustd\cup\Ustdc:\;n-1\text{ or }n\text{ is below the first row of }T\} \\
			&=\Ustd\cup\Ustdc\setminus\MT(\lambda).
		\end{align*}
		Throughout this proof, we use $\mathcal{C}_{\lambda}=e^{-\frac{3}{n}(n-\lambda_1)}\left(1+O\left(\frac{3}{n^2}\right)\right)$ for all $\lambda\vdash n$ satisfying $n-\lambda_1\ll n$, from \cite[Corollary]{N-limit}. Thus, for sufficiently large $n$, using $1+O\left(\frac{3}{n^2}\right)\approx e^{O(3/n^2)}$, we have that
		\begin{equation}\label{eq:Main-term1}
			\mathcal{C}_{\lambda}^{\lceil\frac{n}{3}(\log n+c)\rceil}=e^{o(1)-(n-\lambda_1)(\log n+c)},%=e^{-(n-\lambda_1)(\log n+c)}e^{O\left(\frac{\log n}{n}\right)+O\left(\frac{1}{n}\right)},
		\end{equation}
		for all $\lambda\in\Ncpar\cup\Cpar$ satisfying $n-M\leq\lambda_1<n$. Now,
		\begin{align}
			\mathcal{S}_1(n):=&\sum_{\substack{\lambda\in\Ncpar\cup\Cpar\\n-M\leq \lambda_1<n}}d_{\lambda}\sum_{T\in\MT(\lambda)}\left(\mathcal{C}_{\lambda}^{\lceil\frac{n}{3}(\log n+c)\rceil}-\mathcal{E}_T^{\lceil\left(n-\frac{3}{2}\right)(\log n+c)\rceil}\right)^2 \nonumber\\
			=&\sum_{\substack{\lambda\in\Ncpar\cup\Cpar\\n-M\leq \lambda_1<n}}d_{\lambda}\sum_{T\in\MT(\lambda)}\left(e^{o(1)-(n-\lambda_1)(\log n+c)}-\left(\frac{2\lambda_1-3}{2n-3}\right)^{\lceil\left(n-\frac{3}{2}\right)(\log n+c)\rceil}\right)^2.\label{eq:Main-term2}
		\end{align}
		The equality in \eqref{eq:Main-term2} follows from \eqref{eq:Main-term1} and \Cref{thm:TT2R_eigenvalues_irr}. Also, using
		\[\begin{array}{lr}
			|\MT(\lambda)|\leq\begin{cases}
				d_{\lambda}&\hspace*{-1.5ex}\text{ for non-self-conjugate }\lambda,\\
				\frac{d_{\lambda}}{2}&\hspace*{-1.5ex}\text{ for self-conjugate }\lambda,
			\end{cases}&\hspace*{-1.5ex}\text{ and }\;\left(\frac{2\lambda_1-3}{2n-3}\right)^{\lceil\left(n-\frac{3}{2}\right)(\log n+c)\rceil}\approx e^{-(n-\lambda_1)(\log n+c)},
		\end{array}\]
		the expression in the right hand side of \eqref{eq:Main-term2}, i.e., $\mathcal{S}_1(n)$ is less than 
		\[\displaystyle\sum_{\substack{\lambda\in\Ncpar\cup\Cpar\\n-M\leq \lambda_1<n}}d_{\lambda}^2e^{-2(n-\lambda_1)(\log n+c)}\left(e^{o(1)}-1\right)^2.\]
		Therefore, 
		$\displaystyle\sum_{\substack{\lambda\vdash n\;:\;n-M\leq \lambda_1<n\\\lambda\notin\Ncpar\cup\Cpar}}d_{\lambda}^2e^{-2(n-\lambda_1)(\log n+c)}\left(e^{o(1)}-1\right)^2\geq 0$
		implies that
		\begin{align}
			\mathcal{S}_1(n)<&\sum_{\lambda\vdash n\;:\;n-M\leq \lambda_1<n}d_{\lambda}^2e^{-2(n-\lambda_1)(\log n+c)}\left(e^{o(1)}-1\right)^2 \nonumber\\
			\leq&\sum_{\lambda_1=n-M}^{n-1}\binom{n}{\lambda_1}^2(n-\lambda_1)!\;e^{-2(n-\lambda_1)(\log n+c)}\left(e^{o(1)}-1\right)^2 \label{eq:Main-term3}\\
			%=&\sum_{m=1}^{M}\frac{\prod_{i=0}^{m-1}(n-i)^2}{m!}\;e^{-2m(\log n+c)}\left(e^{o(1)}-1\right)^2,\;\;\text{ writing }m\text{ for }n-\lambda_1,\nonumber\\
			\leq&\sum_{m=1}^{M}\frac{n^{2m}}{m!}\;e^{-2m(\log n+c)}\left(e^{o(1)}-1\right)^2,\;\text{ setting }m=n-\lambda_1\text{
			and using } \binom{n}{m}\leq\frac{n^m}{m!} \nonumber\\
			=&\sum_{m=1}^{M}\frac{\left(e^{-2c}\right)^m}{m!}\left(e^{o(1)}-1\right)^2<\sum_{m=1}^{\infty}\frac{\left(e^{-2c}\right)^m}{m!}\left(e^{o(1)}-1\right)^2=\left(e^{e^{-2c}}-1\right)\left(e^{o(1)}-1\right)^2\label{eq:Main-term4}.
		\end{align}
		The inequality in \eqref{eq:Main-term3} follows from the facts that $d_{\lambda}\leq\binom{n}{\lambda_1}d_{\xi}$ for all $\xi \vdash (n-\lambda_1)$ with $\xi_1\leq\lambda_1$ and $\displaystyle\sum_{\xi \, \vdash \,(n-\lambda_1)}d_{\xi}^2=(n-\lambda_1)!$.
       % , thanks to Robinson-Schensted algorithm \cite[Section 3.1]{Sagan}.
        The expression in the right hand side of \eqref{eq:Main-term4} approaches to zero as $n\rightarrow\infty$. Therefore, using the non negativity of $\mathcal{S}_1(n)$ we have that 
		\begin{equation}\label{eq:Main-term-key-I}
			\lim_{n\rightarrow\infty}\mathcal{S}_1(n)=0.
		\end{equation}
		Again, using the triangle inequality of real numbers and \eqref{eq:Main-term1}, we have that
		\begin{align}
			\mathcal{S}_2(n):=&\sum_{\substack{\lambda\in\Ncpar\cup\Cpar\\n-M\leq \lambda_1<n}}d_{\lambda}\sum_{T\in\ET(\lambda)}\left(\mathcal{C}_{\lambda}^{\lceil\frac{n}{3}(\log n+c)\rceil}-\mathcal{E}_T^{\lceil\left(n-\frac{3}{2}\right)(\log n+c)\rceil}\right)^2 \nonumber\\
			\leq&\sum_{\substack{\lambda\in\Ncpar\cup\Cpar\\n-M\leq \lambda_1<n}}d_{\lambda}\sum_{T\in\ET(\lambda)}\left(e^{o(1)-(n-\lambda_1)(\log n+c)}+\left|\mathcal{E}_T\right|^{\lceil\left(n-\frac{3}{2}\right)(\log n+c)\rceil}\right)^2 \nonumber\\
			\leq&\sum_{\substack{\lambda\in\Ncpar\cup\Cpar\\n-M\leq \lambda_1<n}}d_{\lambda}\sum_{T\in\ET(\lambda)}\left(e^{o(1)-(n-\lambda_1)(\log n+c)}+\left|\frac{2\lambda_1-3}{2n-3}\right|^{\lceil\left(n-\frac{3}{2}\right)(\log n+c)\rceil}\right)^2.\label{eq:Main-term5}
		\end{align}
		The inequality in \eqref{eq:Main-term5} follows from \Cref{thm:TT2R_eigenvalues_irr} and
		\[c(b_T(n))+c(b_T(n-1))\leq 2\lambda_1-3,\;\text{ i.e., }\left|\mathcal{E}_T\right|\leq\left|\frac{2\lambda_1-3}{2n-3}\right|,\]
		for all $\lambda\in\Ncpar\cup\Cpar$. As $n$ is sufficiently large throughout this proof, we have that
		\[\left|\frac{2\lambda_1-3}{2n-3}\right|^{\lceil\left(n-\frac{3}{2}\right)(\log n+c)\rceil}=\left|1-\frac{n-\lambda_1}{n-\frac{3}{2}}\right|^{\lceil\left(n-\frac{3}{2}\right)(\log n+c)\rceil}\approx\;e^{-(n-\lambda_1)(\log n+c)}.\]
		For every $T\in\ET(\lambda),\;\lambda\in\Ncpar\cup\Cpar$ satisfying $n-M\leq\lambda_1<n$, at least one of $n-1$ or $n$ sits at an inner corner of $\lambda$ below the first row. Also, there could be at most $M$ many inner corner below the first row of $\lambda$. Therefore, using \eqref{eq:dimension-ineq}, we have that
		\[|\ET(\lambda)|\leq\begin{cases}
			\frac{M\times4^{n-\lambda_1}}{n}d_{\lambda}\leq \frac{M\times4^{M}}{n}d_{\lambda}&\text{ for non-self-conjugate }\lambda,\\
			\frac{M\times4^{n-\lambda_1}}{n}\frac{d_{\lambda}}{2}\leq \frac{M\times4^{M}}{n}\frac{d_{\lambda}}{2}&\text{ for self-conjugate }\lambda.
		\end{cases}\]
		Thus, the right hand side of  \eqref{eq:Main-term5}, and hence $\mathcal{S}_2(n)$ is less than or equal to
		\begin{align}
			&\sum_{\substack{\lambda\in\Ncpar\cup\Cpar\\n-M\leq \lambda_1<n}}\frac{M\times4^{M}}{n}\;d_{\lambda}^2\;e^{-2(n-\lambda_1)(\log n+c)}\left(e^{o(1)}+1\right)^2 \nonumber\\
			\leq&\sum_{\lambda\vdash n\;:\;n-M\leq \lambda_1<n}\frac{M\times4^{M}}{n}\;d_{\lambda}^2\;e^{-2(n-\lambda_1)(\log n+c)}\left(e^{o(1)}+1\right)^2 \label{eq:Main-term6}\\
			\leq&\sum_{\lambda_1=n-M}^{n-1}\frac{M\times4^{M}}{n}\;\binom{n}{\lambda_1}^2\;(n-\lambda_1)!\;e^{-2(n-\lambda_1)(\log n+c)}\left(e^{o(1)}+1\right)^2 \label{eq:Main-term7}\\
			=&\frac{M\times4^{M}}{n}\sum_{m=1}^{M}\frac{\prod_{i=0}^{m-1}(n-i)^2}{m!}\;e^{-2m(\log n+c)}\left(e^{o(1)}+1\right)^2,\;\;\text{ writing }m\text{ for }n-\lambda_1 \nonumber\\
			\leq&\frac{M\times4^{M}}{n}\sum_{m=1}^{M}\frac{n^{2m}}{m!}\;e^{-2m(\log n+c)}\left(e^{o(1)}+1\right)^2 \nonumber%\\
			%=&\frac{M\times4^{M}}{n}\sum_{m=1}^{M}\frac{\left(e^{-2c}\right)^m}{m!}\left(e^{o(1)}+1\right)^2<\frac{M\times4^{M}}{n}\left(e^{e^{-2c}}-1\right)\left(e^{o(1)}+1\right)^2.\nonumber
		\end{align}
		\[=\frac{M\times4^{M}}{n}\sum_{m=1}^{M}\frac{\left(e^{-2c}\right)^m}{m!}\left(e^{o(1)}+1\right)^2<\frac{M\times4^{M}}{n}\left(e^{e^{-2c}}-1\right)\left(e^{o(1)}+1\right)^2.\]
		The inequality in \eqref{eq:Main-term6} follows from the fact that
		\[\sum_{\substack{\lambda\vdash n\;:\;n-M\leq \lambda_1<n\\\lambda\notin\Ncpar\cup\Cpar}}\frac{M\times4^{M}}{n}\;d_{\lambda}^2\;e^{-2(n-\lambda_1)(\log n+c)}\left(e^{o(1)}+1\right)^2\geq 0.\] 
		The inequality in \eqref{eq:Main-term7} follows from the facts that $d_{\lambda}\leq\binom{n}{\lambda_1}d_{\xi}$ for all $\xi\vdash (n-\lambda_1)$ with $\xi_1\leq\lambda_1$ and $\displaystyle\sum_{\xi\vdash (n-\lambda_1)}d_{\xi}^2=(n-\lambda_1)!$.
        % , thanks to Robinson-Schensted algorithm \cite[Section 3.1]{Sagan}. 
        Thus,
		\[0\leq \mathcal{S}_2(n)\leq\frac{M\times4^{M}}{n}\left(e^{e^{-2c}}-1\right)\left(e^{o(1)}+1\right)^2,
        % \;\text{ for all sufficiently large }n.
        \]
for all sufficiently large $n$, i.e.,
		\begin{equation}
        \label{eq:Main-term-key-II}
\lim_{n\rightarrow\infty}\mathcal{S}_2(n)=0.
		\end{equation}
		Therefore, the lemma follows from \eqref{eq:Main-term-key-I}, \eqref{eq:Main-term-key-II}, and the following:
		\begin{align*}
			0&\leq
			\sum_{\substack{\lambda\in\Ncpar\cup\Cpar\\n-M\leq \lambda_1<n}}d_{\lambda}\sum_{T\in\Ustd\cup\Ustdc}\left(\mathcal{C}_{\lambda}^{\lceil\frac{n}{3}(\log n+c)\rceil}-\mathcal{E}_T^{\lceil\left(n-\frac{3}{2}\right)(\log n+c)\rceil}\right)^2\\
			&\leq\mathcal{S}_1(n)+\mathcal{S}_2(n),\;\text{ for all sufficiently large }n.\qedhere
		\end{align*}
	\end{proof}
	Following the notations of \Cref{prop:Main-term}, we note down an immediate corollary as follows:
	\begin{cor}\label{cor:Main-term}
		There exists large enough positive integer $\bar{N}=\bar{N}(c,\varepsilon,M)>0$ such that,
		\[\sum_{\substack{\lambda\in\Ncpar\cup\Cpar\\n-M< \lambda_1<n}}d_{\lambda}\sum_{T\in\Ustd\cup\Ustdc}\left(\mathcal{C}_{\lambda}^{\lceil\frac{n}{3}(\log n+c)\rceil}-\mathcal{E}_T^{\lceil\left(n-\frac{3}{2}\right)(\log n+c)\rceil}\right)^2<\frac{\varepsilon}{2},\text{ for all }n\geq \bar{N}.\]
	\end{cor}
	We now write the proof of \Cref{main-thm} below.
	\begin{proof}[\textbf{Proof of \Cref{main-thm}}]
		Nestoridi and Olesker-Taylor \cite{N-limit} obtained the limit profile for the random walk on $A_n$ generated by all $3$-cycles. Their result shows
		\[\lim_{n\rightarrow\infty}\left\|Q^{*\lceil\frac{n}{3}\log n+cn\rceil}-U_{A_n}\right\|_{\text{TV}}=d_{\text{TV}}\left(\text{Poi}(1+e^{-c}),\text{Poi}(1)\right),\;c\in\mathbb{R}.\]
		Therefore, the theorem follows from \Cref{thm:comparison} and $\displaystyle\lim_{n\rightarrow\infty}\sum_{\rho\in\widehat{A_n}}\T_{\rho}=0$; where
		\begin{equation}
        \label{eq:key}
			\T_{\rho}:= d_{\rho} \Tr\left(\left(\widehat{Q}(\rho)\right)^{\lceil\frac{n}{3}(\log n+c)\rceil}-\left(\widehat{P}(\rho)\right)^{\lceil\left(n-\frac{3}{2}\right)(\log n+c)\rceil}\right)^2\text{ for }\rho\in\widehat{A_n}.
		\end{equation}
		Recall the irreducible representations of $A_n$ from \eqref{eq:irr-A_n-module}. They are given by
		\[\widehat{A_n}=\big\{D_{\lambda}:\lambda\in\Ncpar\big\}\bigcup\big\{D^+_{\lambda},D^-_{\lambda}:\lambda\in\Cpar\big\}.\]
		We now focus on each summand, given in \eqref{eq:key}, indexed by the irreducible representations on $A_n$. For $\lambda\in\Ncpar$, there is an irreducible representation of $A_n$, and
		\begin{align}\label{eq:key1}
			\T_{\lambda}=&d_{\lambda}\Tr\left(\mathcal{C}_{\lambda}^{\lceil\frac{n}{3}(\log n+c)\rceil}I_{d_{\lambda}}-\left(\widehat{P}(\lambda)\right)^{\lceil\left(n-\frac{3}{2}\right)(\log n+c)\rceil}\right)^2 \nonumber\\
			=&d_\lambda\sum_{T\in\Ustd\cup\Ustdc}\left(\mathcal{C}_{\lambda}^{\lceil\frac{n}{3}(\log n+c)\rceil}-\left(\mathcal{E}_T\right)^{\lceil\left(n-\frac{3}{2}\right)(\log n+c)\rceil}\right)^2,
		\end{align}
		by \Cref{thm:3-cycle_eigenvalues_irr} and  \Cref{thm:TT2R_eigenvalues_irr}. Also, for $\lambda\in\Cpar$, there are two irreducible representations $\lambda^{\pm}$ of $A_n$, and
		\begin{align*}
			\T_{\lambda^{\pm}}=&\frac{d_{\lambda}}{2}\Tr\left(\mathcal{C}_{\lambda}^{\lceil\frac{n}{3}(\log n+c)\rceil}I_{d_{\lambda/2}}-\left(\widehat{P}(\lambda^{\pm})\right)^{\lceil\left(n-\frac{3}{2}\right)(\log n+c)\rceil}\right)^2 \nonumber\\
			=&\frac{d_{\lambda}}{2}\sum_{T\in\Ustd\cup\Ustdc}\left(\mathcal{C}_{\lambda}^{\lceil\frac{n}{3}(\log n+c)\rceil}-\left(\mathcal{E}_T\right)^{\lceil\left(n-\frac{3}{2}\right)(\log n+c)\rceil}\right)^2.
		\end{align*}
		by \Cref{thm:3-cycle_eigenvalues_irr} and \Cref{thm:TT2R_eigenvalues_irr}. Therefore, we have that
		\begin{align}\label{eq:key2}
			&\sum_{\rho\in\{\lambda^+,\lambda^-\}}\T_{\rho}
			=&d_{\lambda}\sum_{T\in\Ustd\cup\Ustdc}\left(\mathcal{C}_{\lambda}^{\lceil\frac{n}{3}(\log n+c)\rceil}-\left(\mathcal{E}_T\right)^{\lceil\left(n-\frac{3}{2}\right)(\log n+c)\rceil}\right)^2.
		\end{align}for $\lambda\in\Cpar$. Thus, combining \eqref{eq:key1} and \eqref{eq:key2} we have that
		\begin{align*}
			0\leq&\sum_{\rho\in\widehat{A_n}}\T_{\rho}
			\leq\hspace*{-1ex}\sum_{\lambda\in\Ncpar\cup\Cpar}\hspace*{-1ex}d_{\lambda}\hspace*{-1ex}\sum_{T\in\Ustd\cup\Ustdc}\hspace*{-1ex}\left(\mathcal{C}_{\lambda}^{\lceil\frac{n}{3}(\log n+c)\rceil}\hspace*{-2ex}-\left(\mathcal{E}_T\right)^{\lceil\left(n-\frac{3}{2}\right)(\log n+c)\rceil}\right)^2\hspace*{-1ex}.\nonumber
		\end{align*}
		The summand corresponding to $\lambda=(n)$ is zero. Thus, \Cref{prop:Error-term} and \Cref{cor:Main-term} together imply
		$0\leq\displaystyle\sum_{\rho\in\widehat{A_n}}\T_{\rho}<\varepsilon$
		for all $n\geq\max\{N,\bar{N}\}$. This completes the proof.
	\end{proof}
    We conclude this section with the following question for further exploration:
    \begin{opqn}
      Determine if the following random walk models exhibit the cutoff phenomenon and, if so, derive their limiting profiles:
        \begin{enumerate}[left=0pt]
            \item The random walk generated by star $k$-cycles for $k=o(n)$, specifically the walks generated by $(1,2,\dots,k-1,i)$ and $(i,k-1,k-2,\dots,2)$ for $k\leq i\leq n$ and $k=o(n)$.
            \item The random walk generated by star conjugacy classes, for example,
            \begin{itemize}
                \item The random walk generated by $(1,2,3)(4,i)$ and $(4,i)(3,2,1)$ for $5\leq i\leq n$, or
                \item The random walk generated by $(1,2,i)(4,5)$ and $(4,5)(i,2,1)$ for $i=3,\;6\leq i\leq n$.
            \end{itemize}
        \end{enumerate} 
    \end{opqn}
%%%%%%%%%%%%%%%%%%%%%%%%%%%%%%%%%%%%%%%%%%%%%%%%%%%%%%%%%%
%%%%%%%%%%%%%%%%%%%%%%%%%%%%%%%%%%%%%%%%%%%%%%%%%%%%%%%%%%
\section{Spectrum of the alternating group graph}\label{sec:graph-spectrum}
 Huang and Huang obtained the second-largest eigenvalues of the \emph{alternating group graph, extended alternating group graph, complete alternating group graph} and asked the questions on the complete spectrum of them. The answer below for the alternating group graph follows directly from \Cref{thm:TT2R_eigenvalues_irr}. The notations in this section have the same meaning as they are in \Cref{sec:spectrum}.

 Recall that the alternating group graph $AG_n$ is the Cayley graph on $A_n$ with generating set $\{(1,i,2),(1,2,i):3\leq i\leq n\}$. Therefore, the adjacency matrix of $AG_n$ is given by the (right) multiplication action of the group algebra element
 \[\displaystyle\sum_{i=3}^{n}\left((1,2,i)+(1,i,2)\right)\in\mathbb{C}[A_n]\]
 on the group algebra $\mathbb{C}[A_n]$. But, the equality
 \begin{align*}
     \sum_{i=3}^{n}\left((1,2,i)+(1,i,2)\right)
     =&(1,n)(2,n-1)\left(\sum_{i=1}^{n-2}\left((n-1,n,i)+(n,n-1,i)\right)\right)(1,n)(2,n-1)
 \end{align*}
 implies that the adjacency matrix of $AG_n$ is similar to $(2n-3)\widehat{P}(R)-I_{\frac{n!}{2}}$. Here, $R$ is the right regular representation of the alternating group and $I_{\frac{n!}{2}}$ denote the identity matrix of size $\frac{n!}{2}\times\frac{n!}{2}$. Therefore, the full spectrum is immediate from \Cref{thm:TT2R_eigenvalues_irr}, and it is formally given below:
 
 \begin{thm}
		For $\lambda\in\Ncpar\cup\Cpar$, the eigenvalues of the adjacency matrix of $AG_n$ are indexed by the set $\Ustd\cup\Ustdc$. For non-self-conjugate $\lambda\vdash n$, each eigenvalue repeats $d_{\lambda}$ many times. For self-conjugate $\lambda\vdash n$, each eigenvalue repeats $\frac{d_{\lambda}}{2}$ many times. Suppose $\mathcal{E'}_T$ denote the eigenvalue indexed by $T\in \Ustd\cup\Ustdc$.
		\begin{itemize}
			\item If $n-1$  and $n$ appear in the same row of $T$, then $\mathcal{E'}_T=c(b_T(n))+c(b_T(n-1))-1$.
			\item If $n-1$  and $n$ appear in the same column of $T$, then $\mathcal{E'}_T=-c(b_T(n))-c(b_T(n-1))-1$.
			\item If $n-1$  and $n$ appear neither in the same row nor in the same row of $T$, then
			\[\begin{cases}
				\mathcal{E'}_T=c(b_T(n))+c(b_T(n-1))-1 \text{ and }\vspace*{1ex}\\
				\mathcal{E'}_S=-c(b_T(n))-c(b_T(n-1))-1.
			\end{cases}\] 
			Here, $S$ is the upper standard Young tableau obtained from $T$ by interchanging the positions of $n$ and $n-1$.
		\end{itemize}
	\end{thm}
 %%%%%%%%%%%%%%%%%%%%%%%%%%%%%%%%%%%%%%%%
 %%%%%%%%%%%%%%%%%%%%%%%%%%%%%%%%%%%%%%%%%%%%%
\section{Example when the transition matrices do not commute}\label{sec:example}
We consider the random walk on $A_n$ driven by the following probability measure:
\begin{equation}\label{eq:TT'R_defn}
	P'(\pi)=
	\begin{cases}
		\frac{1}{(n-1)^2}&\text{ if }\pi=(i,j,n)\text{ or }(j,i,n)\text{ for }1\leq i<j\leq n-1,\\
		%\frac{1}{(n-1)^2}&\text{ if }\pi=(j,i,n-1)\text{ for }1\leq i<j\leq n-1,\\
		\frac{1}{n-1}&\text{ if }\pi=\1,\text{ the identity permutation},\\
		0&\text{ otherwise}.
	\end{cases}
\end{equation}
The transition matrix of this random walk on $A_n$ is the action of $\frac{1}{(n-1)^2}P'_n\in\mathbb{C}[A_n]$, where
\[P'_n=(n-1)\1+\sum_{1\leq i<j\leq n-1}\left((i,j,n)+(j,i,n)\right),\]
on $A_n$ by multiplication on the right. Also, the transition matrix of the transpose top-$2$ with random shuffle is the action of $\frac{1}{2n-3}P_n\in\mathbb{C}[A_n]$, where
\[P_n=\1+\sum_{1\leq i\leq n-2}\left((i,n-1,n)+(i,n,n-1)\right),\]
on $A_n$ by multiplication on the right. It can be checked that $P_nP'_n\neq P'_nP_n$ for all $n\geq 4$. Therefore, the transition matrices $\widehat{P}(R)$ and $\widehat{P'}(R)$ are not simultaneously diagonalizable; where $R$ is the right regular representation of $A_n$. As a result, Nestoridi's comparison technique \cite[Lemma 1.4.]{N-star_limit} fails to compare the random walks on $A_n$ driven by the probability measures $P$ and $P'$. However, we can apply our comparison method (given in \Cref{thm:comparison}) to compare these two random walks. This is because, for each irreducible representation $\rho$ of $A_n$, the matrices $\widehat{P}(\rho)$ and $\widehat{P'}(\rho)$ are simultaneously block-diagonalizable with block size at most $2$.

We now write down the relation between $P_n$ and $P'_n$ with the \emph{Young-Jucys-Murphy} (YJM) elements of $A_n$ \cite{TT2R,Ruff}. The $i$th YJM element $J_i$ of $A_n$ is given by
\[J_i:=(1,2)\left((1,i)+(2,i)+\cdots+(i-1,i)\right),\quad i\geq 3,\]
with $J_1=0$ and $J_2=\1$ \cite[Definition 2.1]{TT2R}. It can be verified that $P_n'=J_n^2$ for all $n\geq 3$. Additionally, recall that 
\[P_n=(1,2)(n-1,n)\left(J_{n-1}+J_n\right),\]
holds for $n>3$ \cite[Lemma 2.1]{TT2R}. We also recall the set of all (non-isomorphic) irreducible $A_n$-modules:
\[\big\{D_{\lambda}:\lambda\in\Ncpar\big\}\bigcup\big\{D^+_{\lambda},D^-_{\lambda}:\lambda\in\Cpar\big\},\]
from \eqref{eq:irr-A_n-module}. We work with the basis used in the proof of \Cref{thm:TT2R_eigenvalues_irr} by the first author (for details, see \cite[Theorem 2.4 and Theorem 2.5]{TT2R}). For $\lambda\in\Ncpar$ (respectively, $\lambda\in\Cpar$) and $T\in\Ustd\cup\Ustdc$, let $v_T$ (respectively, $v_T^{\pm}$) be the basis vector of $D_{\lambda}$ (respectively, $D_{\lambda}^{\pm}$) indexed by $T$, and determined by $J_i(v_T)=c(b_T(i))v_T$ (respectively, $J_i(v_T^{\pm})=c(b_T(i))v_T^{\pm}$) for $1\leq i\leq n$.

For $\lambda\in\Ncpar$, we partition the basis into three subsets, as follows:
\begin{align*}
    &\mathcal{B}_1:=\{v_T: c(b_T(n))=c(b_T(n-1))+1\},\;\mathcal{B}_2:=\{v_T: c(b_T(n))=c(b_T(n-1))-1\},\text{ and }\\
    &\mathcal{B}_3:=\{v_T,(1,2)(n-1,n)v_{T}: c(b_T(n))\neq c(b_T(n-1))\pm 1\text{ and }c(b_T(n-1))<c(b_T(n))\}.
\end{align*}
From \cite[eq. (3)]{TT2R}, we recall that 
\begin{equation*}
    (1,2)(n-1,n)\cdot v_T=\begin{cases}
        v_T&\text{ if }v_T\in\mathcal{B}_1,\\
        -v_T&\text{ if }v_T\in\mathcal{B}_2.
    \end{cases}
\end{equation*}
Therefore, the vectors in $\mathcal{B}_1$ and $\mathcal{B}_2$ are common eigenvectors of both $P_n$ and $P_n'$, with the following eigenvalue relations:
\begin{align}
    &P_n(v_T)=(2c(b_T(n))-1)v_T\text{ and }P_n'(v_T)=\left(c(b_T(n))\right)^2v_T,\text{ if }v_T\in\mathcal{B}_1,\label{eq:eigvcomp1.0}\\
    &P_n(v_T)=-(2c(b_T(n))+1)v_T\text{ and }P_n'(v_T)=\left(c(b_T(n))\right)^2v_T,\text{ if }v_T\in\mathcal{B}_2\label{eq:eigvcomp1.1}.
\end{align}
For $T$ such that $c(b_T(n))\neq c(b_T(n-1))\pm 1$ and $c(b_T(n-1))<c(b_T(n))$, we have that
\[\text{Span}\{\{v_T,v_{(n-1,n)T}\}\}=\text{Span}\{\{v_T,(1,2)(n-1,n)v_T\}\}\;\text{(by \cite[eq. (3)]{TT2R}).}\]
Thus, $\mathcal{B}_3$ has even cardinality, and can be partitioned into subsets of size two, such that the span of each subset is invariant under the action of both $P_n$ and $P_n'$. More specifically, for $n>3$, we have the following:
\begin{align}
    [P_n]_{\{v_T,(1,2)(n-1,n)v_T\}}=&\begin{pmatrix}
    0&c(b_T(n-1))+c(b_T(n))\\&\\
    c(b_T(n-1))+c(b_T(n))&0
\end{pmatrix} \text{ and }\label{eq:eigvcomp2.0}\\&\nonumber\\
[P_n']_{\{v_T,(1,2)(n-1,n)v_T\}}=&\begin{pmatrix}
    \left(c(b_T(n))\right)^2&c(b_T(n))+c(b_T(n-1))\\&\\
    0&\left(c(b_T(n-1))\right)^2
\end{pmatrix}.\label{eq:eigvcomp2.1}
\end{align}

For $\lambda\in\Cpar$, by replacing $D_{\lambda},v_T$, and $\mathcal{B}_i (i=1,2,3)$ with $D^{\pm}_{\lambda},v^{\pm}_T$, and $\mathcal{B}^{\pm}_i (i=1,2,3)$, and using a simiar argument as above, we obtain
\begin{align}
    &P_n(v^{\pm}_T)=(2c(b_T(n))-1)v^{\pm}_T\text{ and }P_n'(v^{\pm}_T)=\left(c(b_T(n))\right)^2v^{\pm}_T,\text{ if }v^{\pm}_T\in\mathcal{B}^{\pm}_1,\label{eq:eigvcomp1.0_pm}\\
    &P_n(v^{\pm}_T)=-(2c(b_T(n))+1)v^{\pm}_T\text{ and }P_n'(v^{\pm}_T)=\left(c(b_T(n))\right)^2v^{\pm}_T,\text{ if }v^{\pm}_T\in\mathcal{B}^{\pm}_2\label{eq:eigvcomp1.1_pm}.
\end{align}
Also, for $T$ such that $c(b_T(n))\neq c(b_T(n-1))\pm 1$ and $c(b_T(n-1))<c(b_T(n))$, we have that
\begin{align}
    [P_n]_{\{v^{\pm}_T,(1,2)(n-1,n)v^{\pm}_T\}}=&\begin{pmatrix}
    0&c(b_T(n-1))+c(b_T(n))\\&\\
    c(b_T(n-1))+c(b_T(n))&0
\end{pmatrix} \text{ and }\label{eq:eigvcomp2.0_pm}\\&\nonumber\\
[P_n']_{\{v^{\pm}_T,(1,2)(n-1,n)v^{\pm}_T\}}=&\begin{pmatrix}
    \left(c(b_T(n))\right)^2&c(b_T(n))+c(b_T(n-1))\\&\\
    0&\left(c(b_T(n-1))\right)^2
\end{pmatrix}.\label{eq:eigvcomp2.1_pm}
\end{align}

We first prove the following lemma before comparing the two random walk models on $A_n$ driven by $P$ and $P'$.
\begin{lem}\label{lem:2x2matrix}
    For $a,b,a',b',\kappa\geq 0$, and positive integers $N_1$ and $N_2$, we have the following:
    \begin{equation}\label{eq:2x2matrix0}
    	\Tr\left(\begin{pmatrix}
    		0&a'+b'\\&\\
    		a'+b'&0
    	\end{pmatrix}^{N_1}\hspace{-2ex}-\hspace{-0.5ex}\begin{pmatrix}
    		a^2&\frac{1}{\kappa}(a+b)\\&\\
    		0&b^2
    	\end{pmatrix}^{N_2}\right)^2\leq a^{4N_2}+b^{4N_2}+2(a'+b')^{2N_1}.
    \end{equation}
\end{lem}
\begin{proof}
    The straightforward application of the principle of mathematical induction on $N_1$ and $N_2$ implies the following:
    \begin{align}
        \begin{pmatrix}
        0&a'+b'\\&\\
        a'+b'&0
    \end{pmatrix}^{N_1} &= (a'+b')^{N_1}\begin{pmatrix}
    \frac{1+(-1)^{N_1}}{2}&\frac{1-(-1)^{N_1}}{2}\\&\\
    \frac{1-(-1)^{N_1}}{2}&\frac{1+(-1)^{N_1}}{2}
\end{pmatrix}\label{eq:2x2matrix1}\\&\nonumber\\
    \begin{pmatrix}
        a^2&\frac{1}{\kappa}(a+b)\\&\\
        0&b^2
    \end{pmatrix}^{N_2}&=\begin{pmatrix}
        a^{2N_2}&\frac{b^{2N_2}-a^{2N_2}}{\kappa(b-a)}\\&\\
        0&b^{2N_2}
    \end{pmatrix}.\label{eq:2x2matrix2}
    \end{align}
    Therefore, for even $N_1$, the fact that the trace of the square of a matrix equals the sum of the squares of its eigenvalues implies that the expression in the left hand side of \eqref{eq:2x2matrix0} is equal to $\left((a'+b')^{N_1}-a^{2N_2}\right)^2+\left((a'+b')^{N_1}-b^{2N_2}\right)^2$. Thus, the lemma follows from
    \[\left((a'+b')^{N_1}-a^{2N_2}\right)^2+\left((a'+b')^{N_1}-b^{2N_2}\right)^2\leq a^{4N_2}+b^{4N_2}+2(a'+b')^{2N_1}.\]
    \indent
    For odd $N_1$, the lemma follows from the following
    \begin{align*}
        &\Tr\left(\begin{pmatrix}
        0&a'+b'\\&\\
        a'+b'&0
    \end{pmatrix}^{N_1}\hspace{-2ex}-\hspace{-0.5ex}\begin{pmatrix}
        a^2&\frac{1}{\kappa}(a+b)\\&\\
        0&b^2
    \end{pmatrix}^{N_2}\right)^2\\&\\
    =&\Tr\begin{pmatrix}
        a^{4N_2}+(a'+b')^{2N_1}-(a'+b')^{N_1}\cdot\frac{b^{2N_2}-a^{2N_2}}{\kappa(b-a)}&\clubsuit\\&\\
        \spadesuit&\hspace*{-5ex}b^{4N_2}+(a'+b')^{2N_1}-(a'+b')^{N_1}\cdot\frac{b^{2N_2}-a^{2N_2}}{\kappa(b-a)}
    \end{pmatrix}\\&\\
    =&a^{4N_2}+b^{4N_2}+2(a'+b')^{2N_1}-2(a'+b')^{N_1}\frac{b^{2N_2}-a^{2N_2}}{\kappa(b-a)}\leq\;a^{4N_2}+b^{4N_2}+2(a'+b')^{2N_1}.\qedhere
    \end{align*}
\end{proof}
We now make a guess for the candidate cutoff time and cutoff window for the random walk on $A_n$ driven by $P'$. 
To do so, we focus on the irreducible representation of $A_n$ indexed by $(n-1,1)$ (or equivalently $(2,1^{n-1})$)), and obtain the eigenvalues of $\widehat{P'}(D_{(n-1,1)})$. Let us denote the elements of $\text{UStd}((n-1,1))\cup\text{UStd}((2,1^{n-1}))$ as follows:
\[T_2:=\begin{array}{c}\young({{\substack{1}}}{{\substack{2}}},{{\substack{3}}},{{\substack{\vdots}}},{{\substack{n}}})\end{array},
\quad T_n:=\begin{array}{c}\young({{\substack{1}}}{{\substack{2}}}{{\substack{\cdots}}}{{\tiny{\substack{n-1}}}},{{\substack{n}}})\end{array},\text{ and }
T_i:=\begin{array}{c}\young({{\substack{1}}}{{\substack{2}}}{{\substack{\cdots}}}{{\substack{n}}},{{\substack{i}}})\end{array}\text{ for }2<i<n.\]
Now, consider the basis partition $\{v_{T_i}:2<i<n-1\},\{v_{T_2}\},\{v_{T_{n-1}},(1,2)(n-1,n)v_{T_{n-1}}\}$. 
Using equations \eqref{eq:eigvcomp1.0}, \eqref{eq:eigvcomp1.1}, and \eqref{eq:eigvcomp2.1}, we can deduce the eigenvalues of $\widehat{P'}(D_{(n-1,1)})$. These eigenvalues are given by:
\begin{itemize}
    \item $\left(\frac{n-2}{n-1}\right)^2$, with multiplicity $n-2$, and
    \item $\left(\frac{1}{n-1}\right)^2$, with multiplicity $1$.
\end{itemize}
Therefore, we have that
\[(n-2)\left(\frac{n-2}{n-1}\right)^{2k}+\left(\frac{1}{n-1}\right)^{2k}=(n-2)\left(1-\frac{1}{n-1}\right)^{2k}+\left(\frac{1}{n-1}\right)^{2k}\approx e^{-c},\]
for $k=\frac{1}{2}(n-1)(\log n+c)$. This indicates that $\frac{1}{2}(n-1)\log n$ would be a possible candidate for the cutoff time, with a window of order $n$.

We now use \Cref{thm:comparison} to compare the random walks on $A_n$ driven by $P$ and $P'$. To begin, let us introduce some notation that we will use throughout the rest of this section. We denote
\begin{equation}\label{eq:P-P'-component}
    \Tal_{\rho}:=\Tr\left(\left(\widehat{P}(\rho)\right)^{\left\lceil\left(n-\frac{3}{2}\right)(\log n+c)\right\rceil}-\left(\widehat{P'}(\rho)\right)^{\left\lceil\frac{1}{2}(n-1)(\log n+c)\right\rceil}\right)^2,\text{ for }\rho\in\widehat{A_n}.
\end{equation}
In particular, when $\rho=D_{\lambda},\;\lambda\in\Ncpar$, we simply write $\lambda$ in place of $D_{\lambda}$. Similarly, when $\rho=D^{\pm}_{\lambda},\;\lambda\in\Cpar$, we write $\lambda^{\pm}$ in place of $D^{\pm}_{\lambda}$. Now, we consider the following sum:
\begin{equation}\label{eq:P-P'-comp}
            \mathscr{S}um(n):=\sum_{\rho\in\widehat{A_n}}d_{\rho}
            \Tal_{\rho}=\hspace{-1ex}\sum_{\lambda\in\Ncpar}\hspace{-2ex}d_{\lambda}
            \Tal_{\lambda}+\hspace{-1ex}\sum_{\lambda\in\Cpar}\hspace{-2ex}\left(d_{\lambda^+}
            \Tal_{\lambda^+}+d_{\lambda^-}
            \Tal_{\lambda^-}\right),
\end{equation}
and prove the following lemma.
\begin{lem}\label{lem:one_summand-NC}
    Let $\lambda\in\Ncpar$, and set $\alpha:=\left(n-\frac{3}{2}\right)(\log n+c),\;\beta:=\frac{1}{2}(n-1)(\log n+c)$. Then, $\Tal_{\lambda}=\Tr\left(\left(\widehat{P}(\lambda)\right)^{\left\lceil\alpha\right\rceil}-\left(\widehat{P'}(\lambda)\right)^{\left\lceil\beta\right\rceil}\right)^2\nonumber$ is less than or equal to the following:
    \begin{align*}
        &\sum_{\substack{T\in\Ustd\cup\Ustdc\\c(b_T(n))=c(b_T(n-1))+1}}\hspace{-2ex}\left(\left(\frac{c(b_T(n))+c(b_T(n-1))}{2n-3}\right)^{\left\lceil\alpha\right\rceil}-\left(\frac{c(b_T(n))}{n-1}\right)^{2\left\lceil\beta\right\rceil}\right)^2\\
        &+\hspace{-1ex}\sum_{\substack{T\in\Ustd\cup\Ustdc\\c(b_T(n))=c(b_T(n-1))-1}}\hspace{-2ex}\left(\left(-\frac{c(b_T(n))+c(b_T(n-1))}{2n-3}\right)^{\left\lceil\alpha\right\rceil}-\left(\frac{c(b_T(n))}{n-1}\right)^{2\left\lceil\beta\right\rceil}\right)^2\nonumber\\
        &+\hspace{-1ex}\sum_{\substack{T\in\Ustd\cup\Ustdc\\c(b_T(n))\neq c(b_T(n-1))\pm1\\c(b_T(n-1))<c(b_T(n))}}\hspace{-2ex}\left(\left(\frac{c(b_T(n))}{n-1}\right)^{4\left\lceil\beta\right\rceil}\hspace{-2ex}+\left(\frac{c(b_T(n-1))}{n-1}\right)^{4\left\lceil\beta\right\rceil}\hspace{-2ex}+2\left(\frac{c(b_T(n))+c(b_T(n-1))}{2n-3}\right)^{2\left\lceil\alpha\right\rceil}\right).
    \end{align*}
    Moreover, for $\lambda\in\Cpar$, the same conclusion holds with $\lambda$ replaced by $\lambda^{\pm}$.
\end{lem}
\begin{proof}
    Let us denote $a'_T:=\frac{c(b_T(n))}{2n-3}, b'_T:=\frac{c(b_T(n-1))}{2n-3}, a_T:=\frac{c(b_T(n))}{(n-1)}$, and $b_T:=\frac{c(b_T(n-1))}{(n-1)}$. Then, using equations \eqref{eq:eigvcomp1.0},\eqref{eq:eigvcomp1.1}, \eqref{eq:eigvcomp2.0}, and \eqref{eq:eigvcomp2.1}, we can write $\Tal_{\lambda}$ as follows:
   \begin{align*}
   &\sum_{\substack{T\in\Ustd\cup\Ustdc\\c(b_T(n))=c(b_T(n-1))+1}}\hspace{-2ex}\left(\left(a'_T+b'_T\right)^{\left\lceil\alpha\right\rceil}-\left(a_T\right)^{2\left\lceil\beta\right\rceil}\right)^2+\hspace{-1ex}\sum_{\substack{T\in\Ustd\cup\Ustdc\\c(b_T(n))=c(b_T(n-1))-1}}\hspace{-2ex}\left(\left(-a'_T-b'_T\right)^{\left\lceil\alpha\right\rceil}-\left(a_T\right)^{2\left\lceil\beta\right\rceil}\right)^2\nonumber\\
    &\quad+\sum_{\substack{T\in\Ustd\cup\Ustdc\\c(b_T(n))\neq c(b_T(n-1))\pm1\\c(b_T(n-1))<c(b_T(n))}}\Tr\left(
    \begin{pmatrix}
        0&a'_T+b'_T\\&\\
        a'_T+b'_T&0
    \end{pmatrix}^{\left\lceil\alpha\right\rceil}-
    \begin{pmatrix}
        a^2_T&\frac{1}{(n-1)}(a_T+b_T)\\&\\
        0&b^2_T
    \end{pmatrix}^{\left\lceil\beta\right\rceil}\right)^2.
    \end{align*}
    Thus, the lemma follows directly from \Cref{lem:2x2matrix}.\\
    \indent
    Moreover, for $\lambda\in\Cpar$, the same conclusion holds by substituting $\lambda$ with $\lambda^{\pm}$, and using the equations \eqref{eq:eigvcomp1.0_pm},\eqref{eq:eigvcomp1.1_pm}, \eqref{eq:eigvcomp2.0_pm}, \eqref{eq:eigvcomp2.1_pm}, along with \Cref{lem:2x2matrix}.
    \end{proof}
\begin{lem}\label{lem:Error-term'}
    Given any $\varepsilon>0$ and $c\in\mathbb{R}$, there exist constants $M=M(c,\varepsilon)>0$ and sufficiently large $N=N(c,\varepsilon,M)>M$ such that for all $n\geq N$, we have the following bound on $\mathscr{S}um_1(n)$:
    \[\mathscr{S}um_1(n)=\sum_{\substack{\lambda\in\Ncpar\\\lambda_1\leq n-M}}d_{\lambda}\Tal_{\lambda}+\sum_{\substack{\lambda\in\Cpar\\\lambda_1\leq n-M}}d^+_{\lambda}\Tal_{\lambda^+}+\sum_{\substack{\lambda\in\Cpar\\\lambda_1\leq n-M}}d^-_{\lambda}\Tal_{\lambda^-}<\frac{\varepsilon}{2}.\]
	\end{lem}
    \begin{proof}
        The proof is similar to that of \Cref{prop:Error-term}. For $\lambda\in\Ncpar\cup\Cpar$ and $T\in\Ustd\cup\Ustdc$, we have $\left|\frac{c(b_T(n))+c(b_T(n-1))}{2n-3}\right|\leq \frac{2\lambda_1-3}{2n-3}\leq 1,\;\left|\frac{c(b_T(n-1))}{n-1}\right|\leq \frac{\lambda_1-1}{n-1}\leq 1$, and $\left|\frac{c(b_T(n))}{n-1}\right|\leq \frac{\lambda_1-1}{n-1}\leq 1$. Recall $\alpha:=\left(n-\frac{3}{2}\right)(\log n+c)$ and $\beta:=\frac{1}{2}(n-1)(\log n+c)$. Then, for 
    \[\rho\in\{D_{\lambda}:\lambda\in\Ncpar\}\bigcup\{D^-_{\lambda},D^+_{\lambda}:\lambda\in\Cpar\},\]
    using $(a-b)^2\leq 2(a^2+b^2)$ for $a,b\in\mathbb{R}$, \Cref{lem:one_summand-NC} implies 
    \begin{align}
    \Tal_{\rho}\leq&\;2\sum_{\substack{T\in\Ustd\cup\Ustdc\\c(b_T(n))=c(b_T(n-1))\pm1}}\left(\left(\frac{2\lambda_1-3}{2n-3}\right)^{2\left\lceil\alpha\right\rceil}+\left(\frac{\lambda_1-1}{n-1}\right)^{4\left\lceil\beta\right\rceil}\right)\nonumber\\
    &\quad\quad+\sum_{\substack{T\in\Ustd\cup\Ustdc\\c(b_T(n))\neq c(b_T(n-1))\pm1\\c(b_T(n-1))<c(b_T(n))}}2\left(\left(\frac{\lambda_1-1}{n-1}\right)^{4\left\lceil\beta\right\rceil}+\left(\frac{2\lambda_1-3}{2n-3}\right)^{2\left\lceil\alpha\right\rceil}\right) \nonumber\\
    \leq&\;2\sum_{T\in\Ustd\cup\Ustdc}\left(\left(\frac{2\lambda_1-3}{2n-3}\right)^{2\left\lceil\alpha\right\rceil}+\left(\frac{\lambda_1-1}{n-1}\right)^{4\left\lceil\beta\right\rceil}\right) \label{eq:Error-term'2.2}\\
    \leq &\;2\sum_{T\in\Ustd\cup\Ustdc}\left(\left(\frac{2\lambda_1-3}{2n-3}\right)^{2\alpha}+\left(\frac{\lambda_1-1}{n-1}\right)^{4\beta}\right),\text{ using }\alpha\leq\lceil\alpha\rceil\text{ and }\beta\leq\lceil\beta\rceil \nonumber\\
    \leq &\;2\sum_{T\in\Ustd\cup\Ustdc}\left(e^{-\frac{4\alpha(n-\lambda_1)}{2n-3}}+e^{-\frac{4\beta(n-\lambda_1)}{n-1}}\right),\text{ using }1-x\leq e^{-x}\text{ for all }x\geq 0 \nonumber\\
    \leq &\;2\sum_{T\in\Ustd\cup\Ustdc}2e^{-2(n-\lambda_1)(\log n+c)},\text{ writing the values of }\alpha\text{ and }\beta \nonumber\\
    =&\begin{cases}
        d_{\lambda}\times 4e^{-2(n-\lambda_1)(\log n+c)}&\text{ if }\rho=D_{\lambda},\\
        d^{\pm}_{\lambda}\times 4e^{-2(n-\lambda_1)(\log n+c)}&\text{ if }\rho=D^{\pm}_{\lambda}.
    \end{cases}\label{eq:Error-term'2.3}
    \end{align}
    The inequality in \eqref{eq:Error-term'2.2} follows from $\left(\frac{2\lambda_1-3}{2n-3}\right)^2\leq 1,\left(\frac{\lambda_1-1}{n-1}\right)^4\leq 1,\alpha\leq\lceil\alpha\rceil,\beta\leq\lceil\beta\rceil$, and the following facts:
    \begin{itemize}
        \item $\Ustd\cup\Ustdc$ is a disjoint union of $\{T:c(b_T(n))= c(b_T(n-1))\pm1\}$ and $\{T:c(b_T(n))\neq c(b_T(n-1))\pm1\}$.
        \item $\left|\Big\{T:c(b_T(n))\neq c(b_T(n-1))\pm1\Big\}\right|=2\left|\Big\{T:
        \begin{minipage}{2in}
            $c(b_T(n))\neq c(b_T(n-1))\pm1,\\c(b_T(n-1))< c(b_T(n))$
        \end{minipage}
    \Big\}\right|$.
    \end{itemize}
    We rewrite the following from \eqref{eq:Error-term'2.3}:
    \[\Tal_{\rho}=\Tr\left(\left(\widehat{P}(\rho)\right)^{\left\lceil\alpha\right\rceil}-\left(\widehat{P'}(\rho)\right)^{\left\lceil\beta\right\rceil}\right)^2\leq \begin{cases}
        d_{\lambda}\times 4e^{-2(n-\lambda_1)(\log n+c)}&\text{ if }\rho=D_{\lambda},\\
        d^{\pm}_{\lambda}\times 4e^{-2(n-\lambda_1)(\log n+c)}&\text{ if }\rho=D^{\pm}_{\lambda}.
    \end{cases}\]
    Now, choose a sufficiently large positive integer $M=M(c,\varepsilon)$ such that $\displaystyle\sum_{m\geq M}\frac{e^{-2mc}}{m!}<\frac{\varepsilon}{8}$. Then, for $n\geq N=M+1$, we have that
    \begin{align}
        \mathscr{S}um_1(n)\leq &\sum_{\substack{\lambda\in\Ncpar\\\lambda_1\leq n-M}}4d_{\lambda}^2e^{-2(n-\lambda_1)(\log n+c)}+\sum_{\substack{\lambda\in\Cpar\\\lambda_1\leq n-M}}4\left((d^+_{\lambda})^2+(d^-_{\lambda})^2\right)e^{-2(n-\lambda_1)(\log n+c)}\nonumber\\
        <&\sum_{\substack{\lambda\in\Ncpar\\\lambda_1\leq n-M}}4d_{\lambda}^2e^{-2(n-\lambda_1)(\log n+c)}+\sum_{\substack{\lambda\in\Cpar\\\lambda_1\leq n-M}}4d_{\lambda}^2e^{-2(n-\lambda_1)(\log n+c)}\label{eq:Error-term'3.1}\\
        \leq&\sum_{\lambda\vdash n:\;\lambda_1\leq n-M}4d_{\lambda}^2e^{-2(n-\lambda_1)(\log n+c)}\label{eq:Error-term'3.2}\\
        \leq\;&4\sum_{m=M}^{n-1}\frac{e^{-2mc}}{m!}<4\sum_{m\geq M}\frac{e^{-2mc}}{m!}<\frac{\varepsilon}{2}.\label{eq:Error-term'3.3}
    \end{align}
    The inequality in \eqref{eq:Error-term'3.1} follows from the fact that $(d^+_{\lambda})^2+(d^-_{\lambda})^2=2\left(\frac{d_{\lambda}}{2}\right)^2<d_{\lambda}^2$. The inequality in \eqref{eq:Error-term'3.2} follows from the fact that $\displaystyle\sum_{\substack{\lambda\notin\Ncpar\cup\Cpar\\\lambda_1\leq n-M}}4d_{\lambda}^2e^{-2(n-\lambda_1)(\log n+c)}\geq 0$. Finally, the first inequality in \eqref{eq:Error-term'3.3} follows from the same reasoning as in \eqref{eq:Error-term5} -- \eqref{eq:Error-term7}. This completes the proof.
    \end{proof}
    \begin{lem}\label{lem:Main-term'}
		Given any $\varepsilon>0$ and $c\in\mathbb{R}$, recall $M$ from \Cref{lem:Error-term'}. Then, we have the following: \[\displaystyle\lim_{n\rightarrow\infty}\mathscr{S}um_2(n)=0,\]
	\[
   \text{ where }     \mathscr{S}um_2(n)=\sum_{\substack{\lambda\in\Ncpar\\n-M\leq\lambda_1< n}}d_{\lambda}\Tal_{\lambda}+\sum_{\substack{\lambda\in\Cpar\\n-M\leq\lambda_1< n}}d^+_{\lambda}\Tal_{\lambda^+}+\sum_{\substack{\lambda\in\Cpar\\n-M\leq\lambda_1< n}}d^-_{\lambda}\Tal_{\lambda^-}.
	\]
	\end{lem}
    \begin{proof}
        For $\lambda\vdash n$, we recall the partition of $\Ustd\cup\Ustdc$ from the proof of \Cref{prop:Main-term}. More precisely, the set $\Ustd\cup\Ustdc$ is partitioned into two subsets as follows:
		\begin{align*}
			\MT(\lambda)&=\{T\in\Ustd\cup\Ustdc:\text{ both }n-1\text{ and }n\text{ are in the first row of }T\}.\\
			\ET(\lambda)&=\{T\in\Ustd\cup\Ustdc:\;n-1\text{ or }n\text{ is below the first row of }T\}\\
			&=\Ustd\cup\Ustdc\setminus\MT(\lambda).
		\end{align*}
        For the rest of this proof, we work with  $\lambda\in\Ncpar\cup\Cpar$ satisfying $n-M\leq\lambda_1<n$. From the proof of \Cref{prop:Main-term}, we recall the following:
        \begin{align}
            |\MT(\lambda)|&\leq\begin{cases}
				d_{\lambda}&\hspace*{-1.5ex}\text{ for non-self-conjugate }\lambda,\\
				\frac{d_{\lambda}}{2}&\hspace*{-1.5ex}\text{ for self-conjugate }\lambda,
			\end{cases}\label{eq:Main-term'1.1}\\
            |\ET(\lambda)|&\leq\begin{cases}
			\frac{M\times4^{n-\lambda_1}}{n}d_{\lambda}\leq \frac{M\times4^{M}}{n}d_{\lambda}&\text{ for non-self-conjugate }\lambda,\\
			\frac{M\times4^{n-\lambda_1}}{n}\frac{d_{\lambda}}{2}\leq \frac{M\times4^{M}}{n}\frac{d_{\lambda}}{2}&\text{ for self-conjugate }\lambda.
		\end{cases}\label{eq:Main-term'1.2}
        \end{align}
   		From \Cref{lem:one_summand-NC}, let us recall the notations $\alpha=\left(n-\frac{3}{2}\right)(\log n+c),\beta=\frac{1}{2}(n-1)(\log n+c)$, $a'_T=\frac{c(b_T(n))}{2n-3}, b'_T=\frac{c(b_T(n-1))}{2n-3},a_T=\frac{c(b_T(n))}{(n-1)}$, and $b_T=\frac{c(b_T(n-1))}{(n-1)}$. Therefore, using \Cref{lem:one_summand-NC}, for $\rho\in\{\lambda:\lambda\in\Ncpar\}\cup\{\lambda^+,\lambda^-:\lambda\in\Cpar\}$, we get that
        \begin{align}
        \Tal_{\rho}\leq&\hspace{-2ex}\sum_{\substack{T\in\Ustd\cup\Ustdc\\c(b_T(n))=c(b_T(n-1))+1}}\hspace{-2ex}\left(\left(a'_T+b'_T\right)^{\left\lceil\alpha\right\rceil}
        -\left(a_T\right)^{2\left\lceil\beta\right\rceil}\right)^2+\hspace{-4ex}
        \sum_{\substack{T\in\Ustd\cup\Ustdc\\c(b_T(n))=c(b_T(n-1))-1}}\hspace{-3ex}\left(\left(-a'_T-b'_T\right)^{\left\lceil\alpha\right\rceil}-\left(a_T\right)^{2\left\lceil\beta\right\rceil}\right)^2\nonumber\\
        &\quad\quad+\sum_{\substack{T\in\Ustd\cup\Ustdc\\c(b_T(n))\neq c(b_T(n-1))\pm1
        \\c(b_T(n-1))<c(b_T(n))}}\left((a_T)^{4\left\lceil\beta\right\rceil}+(b_T)^{4\left\lceil\beta\right\rceil}
        +2(a'_T+b'_T)^{2\left\lceil\alpha\right\rceil}\right) \nonumber\\
        \leq&\sum_{T\in\MT(\lambda)}\hspace{-2ex}\left(\left(a'_T+b'_T\right)^{\left\lceil\alpha\right\rceil}-\left(a_T\right)^{2\left\lceil\beta\right\rceil}\right)^2+\hspace{-4ex}\sum_{\substack{T\in\ET(\lambda)\\c(b_T(n))=c(b_T(n-1))\pm1}}\hspace{-4ex}\left(|a'_T+b'_T|^{\left\lceil\alpha\right\rceil}+|a_T|^{2\left\lceil\beta\right\rceil}\right)^2\label{eq:Main-term'2}\\
        &\quad\quad+\sum_{\substack{T\in\ET(\lambda)\\c(b_T(n))\neq c(b_T(n-1))\pm1\\c(b_T(n-1))<c(b_T(n))}}\left((a_T)^{4\left\lceil\beta\right\rceil}+(b_T)^{4\left\lceil\beta\right\rceil}+2(a'_T+b'_T)^{2\left\lceil\alpha\right\rceil}\right).\nonumber
        \end{align}
        The inequality in \eqref{eq:Main-term'2} follows from the triangle inequality and the fact that $\MT(\lambda)$ is a subset of 
        $\{T\in\Ustd\cup\Ustdc:c(b_{T}(n))=c(b_{T}(n-1))+1\}$. Thus, we have $a'_T+b'_T=\frac{2\lambda_1-3}{2n-3}$ and $a_T=\frac{\lambda_1-1}{n-1}$ for $T\in\MT(\lambda)$. Additionally, $|a'_T+b'_T|\leq \left|\frac{2\lambda_1-3}{2n-3}\right|\leq 1$ and $|a_T|\leq \left|\frac{\lambda_1-1}{n-1}\right|\leq 1$ for $T\in\ET(\lambda)$. Hence, we have the following:
        \begin{align}\label{eq:Main-term'3}
           (a'_T+b'_T)^{\lceil\alpha\rceil}-(a_T)^{2\lceil\beta\rceil}&=\left(\frac{2\lambda_1-3}{2n-3}\right)^{\lceil\left(n-\frac{3}{2}\right)(\log n+c)\rceil}-\left(\frac{\lambda_1-1}{n-1}\right)^{2\lceil\frac{1}{2}(n-1)(\log n+c)\rceil}\\
           &=e^{-(n-\lambda_1)(\log n+c)}o(1),\text{ for }T\in\MT(\lambda).\nonumber
        \end{align}
        Also, using $\alpha\leq\lceil\alpha\rceil$ and $\alpha\leq\lceil\beta\rceil$, we obtain the following:
        \begin{align}
            &|a'_T+b'_T|^{\lceil\alpha\rceil}\leq \left(\frac{2\lambda_1-3}{2n-3}\right)^{\lceil\left(n-\frac{3}{2}\right)(\log n+c)\rceil}\leq \left(\frac{2\lambda_1-3}{2n-3}\right)^{\left(n-\frac{3}{2}\right)(\log n+c)}\leq e^{-(n-\lambda_1)(\log n+c)},\label{eq:Main-term'4.1}\\
            &|a_T|^{2\lceil\beta\rceil},|b_T|^{2\lceil\beta\rceil}\leq \left(\frac{\lambda_1-1}{n-1}\right)^{2\lceil\frac{n-1}{2}(\log n+c)\rceil}\leq \left(\frac{\lambda_1-1}{n-1}\right)^{(n-1)(\log n+c)}\leq e^{-(n-\lambda_1)(\log n+c)},\label{eq:Main-term'4.2}
        \end{align}
        for $T\in\ET(\lambda)$. Therefore, using the estimates from \eqref{eq:Main-term'3}, \eqref{eq:Main-term'4.1}, and \eqref{eq:Main-term'4.2} in \eqref{eq:Main-term'2}, we get that
        \begin{align}
            \Tal_{\rho}\leq&\sum_{T\in\MT(\lambda)}\left(e^{-(n-\lambda_1)(\log n+c)}o(1)\right)^2+\sum_{\substack{T\in\ET(\lambda)\\c(b_T(n))=c(b_T(n-1))\pm1}}\left(2e^{-(n-\lambda_1)(\log n+c)}\right)^2\nonumber\\
            &\quad+\sum_{\substack{T\in\ET(\lambda)\\c(b_T(n))\neq c(b_T(n-1))\pm1\\c(b_T(n-1))<c(b_T(n))}}\left(e^{-2(n-\lambda_1)(\log n+c)}+e^{-2(n-\lambda_1)(\log n+c)}+2e^{-2(n-\lambda_1)(\log n+c)}\right)\nonumber\\
            \leq&\sum_{T\in\MT(\lambda)}e^{-2(n-\lambda_1)(\log n+c)}\left(o(1)\right)^2+\sum_{T\in\ET(\lambda)}\left(2e^{-(n-\lambda_1)(\log n+c)}\right)^2
            \label{eq:Main-term'5.1}\\
            \leq&\begin{cases}
                \left(\left(o(1)\right)^2+\frac{M\times 4^{M+1}}{n}\right)d_{\lambda}e^{-2(n-\lambda_1)(\log n+c)}&\text{ if }\rho=D_{\lambda},\lambda\text{ is non-self-conjugate},\\
                \left(\left(o(1)\right)^2+\frac{M\times 4^{M+1}}{n}\right)\frac{d_{\lambda}}{2}e^{-2(n-\lambda_1)(\log n+c)}&\text{ if }\rho=D^{\pm}_{\lambda},\lambda\text{ is self-conjugate}.
            \end{cases}\label{eq:Main-term'5.2}
        \end{align}
        The inequality in \eqref{eq:Main-term'5.1} follows from the following facts:
        \begin{itemize}
            \item $\ET(\lambda)$ is a disjoint union of $\{T\in\ET(\lambda):c(b_T(n))= c(b_T(n-1))\pm1\}$ and $\{T\in\ET(\lambda):c(b_T(n))\neq c(b_T(n-1))\pm1\}$.
            \item The set $\Big\{T\in\ET(\lambda):c(b_T(n))\neq c(b_T(n-1))\pm1\Big\}$ contains the set  $\\\Big\{T\in\ET(\lambda):c(b_T(n))\neq c(b_T(n-1))\pm1,\;c(b_T(n-1))< c(b_T(n)) \Big\}$.
        \end{itemize}
        The inequality in \eqref{eq:Main-term'5.2} follows from the bounds in \eqref{eq:Main-term'1.1} and \eqref{eq:Main-term'1.2}, which provide estimates for the number of elements in $\MT(\lambda)$ and $\ET(\lambda)$. Now, using the estimates obtained in \eqref{eq:Main-term'5.2} and the fact that $d_{\lambda}^+=d_{\lambda}^-=\frac{d_{\lambda}}{2}$, we get that
        \begin{align}
            \mathscr{S}um_2(n)&\leq \sum_{\substack{\lambda\in\Ncpar\\n-M\leq\lambda_1<n}}d_{\lambda}^2\left(\left(o(1)\right)^2+\frac{M\times 4^{M+1}}{n}\right)e^{-2(n-\lambda_1)(\log n+c)}\nonumber\\
            &\quad+2\sum_{\substack{\lambda\in\Cpar\\n-M\leq\lambda_1<n}}\left(\frac{d_{\lambda}}{2}\right)^2\left(\left(o(1)\right)^2+\frac{M\times 4^{M+1}}{n}\right)e^{-2(n-\lambda_1)(\log n+c)}\nonumber\\
            &<\sum_{\lambda\vdash n\;:\;n-M\leq\lambda_1<n}d_{\lambda}^2\left(\left(o(1)\right)^2+\frac{M\times 4^{M+1}}{n}\right)e^{-2(n-\lambda_1)(\log n+c)}.\label{eq:Main-term'6}
        \end{align}
        The inequality in \eqref{eq:Main-term'6} follows from the fact that: $\frac{d^2_{\lambda}}{2}<d_{\lambda}^2$, and
        \[\sum_{\substack{\lambda\notin\Ncpar\cup\Cpar\\n-M\leq\lambda_1<n}}d_{\lambda}^2\left(\left(o(1)\right)^2+\frac{M\times 4^{M+1}}{n}\right)e^{-2(n-\lambda_1)(\log n+c)}\geq 0.\]
        We now perform similar calculations to those done in \eqref{eq:Main-term3}-\eqref{eq:Main-term4} but replacing $(o(1))^2$ in place of $(e^{o(1)}-1)^2$. This gives the following inequality:
        \begin{equation}\label{eq:Main-term'7}
            \sum_{\lambda\vdash n\;:\;n-M\leq\lambda_1<n}d_{\lambda}^2e^{-2(n-\lambda_1)(\log n+c)}\left(o(1)\right)^2<\left(e^{e^{-2c}}-1\right)\left(o(1)\right)^2.
        \end{equation}
        Next, we perform similar calculations to those in \eqref{eq:Main-term6}-\eqref{eq:Main-term7}, but replacing $(e^{o(1)}+1)^2$ with $4$, and obtain:
        \begin{equation}\label{eq:Main-term'8}
            \sum_{\lambda\vdash n\;:\;n-M\leq\lambda_1<n}\frac{M\times 4^M}{n}\;d_{\lambda}^2e^{-2(n-\lambda_1)(\log n+c)}\times 4<\frac{M\times 4^M}{n}\left(e^{e^{-2c}}-1\right)\times 4.
        \end{equation}
        Finally, combining the inequalities from \eqref{eq:Main-term'6},\eqref{eq:Main-term'7}, and \eqref{eq:Main-term'8}, we arrive at the following:
        \[0\leq\mathscr{S}um_2(n)<\left(e^{e^{-2c}}-1\right)\left((o(1))^2+\frac{M\times 4^{M+1}}{n}\right).\]
        Thus, the lemma follows from the fact
        $\left(e^{e^{-2c}}-1\right)\left((o(1))^2+\frac{M\times 4^{M+1}}{n}\right)\rightarrow 0$ as $n\rightarrow\infty$.
    \end{proof}
    Following the notations of \Cref{lem:Main-term'}, we note down an immediate corollary as follows:
    \begin{cor}\label{cor:Main-term'}
    There exists a large enough positive integer $\overline{N}=\overline{N}(c,\varepsilon,M)>0$ such that for all $n\geq \overline{N}$, we have $\mathscr{S}um_2(n)<\frac{\varepsilon}{2}$.
    \end{cor}
   Using \Cref{lem:Error-term'} and \Cref{cor:Main-term'}, we obtain that $\mathscr{S}um(n)=\mathscr{S}um_1(n)+\mathscr{S}um_2(n)<\varepsilon$ for all $n\geq \max\{N,\overline{N}\}$, where $\mathscr{S}um(n)$ is given in \eqref{eq:P-P'-comp}. Since $\varepsilon>0$ is arbitrary, we conclude that $\displaystyle\lim_{n\rightarrow\infty}\mathscr{S}um(n)=0$. Therefore, using \Cref{thm:comparison} and \Cref{main-thm}, we conclude that:
    \[\lim_{n\rightarrow\infty}\left\|(P')^{*\lceil\frac{1}{2}(n-1)\log n+cn\rceil}-U_{A_n}\right\|_{\text{TV}}=d_{\text{TV}}\left(\text{Poi}(1+e^{-c}),\text{Poi}(1)\right),\;c\in\mathbb{R}.\]
\begin{rem}
   Although Nestoridi's comparison method fails to compare the random walks on $A_n$ driven by $P'$ and $P$, it applies nicely to the comparison between the random walks on $A_n$ driven by $P'$ and $Q$. The proof technique for comparing these latter walks uses symmetric group character estimates and proceeds similarly to the argument presented in \Cref{sec:main-proof}. Here, we directly compare the random walks on $A_n$ driven by $P'$ and $P$ to illustrate our comparison method (\Cref{thm:comparison}) for random walks with a non-commuting transition matrix.
\end{rem}
 %%----------Acknowledgement---------
\subsection*{Acknowledgement} We sincerely thank the anonymous referees for their valuable comments, which have significantly improved the quality of this article. We also extend our gratitude to Evita Nestoridi and Raghavendra Tripathi for their insightful feedback on the previous version of the manuscript. SG acknowledges the support of the INSPIRE FACULTY FELLOW research grant (IFA 23 MA 198).	
%----------REFERENCES---------
\bibliography{TT2R_L-P_ref}{}

\begin{thebibliography}{10}

\bibitem{AD1}
David Aldous and Persi Diaconis.
\newblock Shuffling cards and stopping times.
\newblock {\em Amer. Math. Monthly}, 93(5):333--348, 1986.

\bibitem{AD2}
David Aldous and Persi Diaconis.
\newblock Strong uniform times and finite random walks.
\newblock {\em Adv. in Appl. Math.}, 8(1):69--97, 1987.

\bibitem{Riff-sh}
Dave Bayer and Persi Diaconis.
\newblock Trailing the dovetail shuffle to its lair.
\newblock {\em Ann. Appl. Probab.}, 2(2):294--313, 1992.

\bibitem{ASEP-limit}
Alexey Bufetov and Peter Nejjar.
\newblock Cutoff profile of {ASEP} on a segment.
\newblock {\em Probab. Theory Related Fields}, 183(1-2):229--253, 2022.

\bibitem{AGG_4}
Eddie Cheng and Marc~J. Lipman.
\newblock Vulnerability issues of star graphs, alternating group graphs and split-stars: strength and toughness.
\newblock {\em Discrete Appl. Math.}, 118(3):163--179, 2002.

\bibitem{delhaye2024brownian}
Jean Delhaye.
\newblock Brownian motion on the unitary quantum group: construction and cutoff.
\newblock {\em arXiv preprint arXiv:2409.06552}, 2024.

\bibitem{D1}
Persi Diaconis.
\newblock Applications of non-commutative fourier analysis to probability problems.
\newblock In {\em {\'E}cole d'{\'E}t{\'e} de Probabilit{\'e}s de Saint-Flour XV--XVII, 1985--87}, pages 51--100. Springer, 1988.

\bibitem{D2}
Persi Diaconis.
\newblock Group representations in probability and statistics.
\newblock 11:vi+198, 1988.

\bibitem{diaconis1990asymptotic}
Persi Diaconis, Ronald~L. Graham, and John~A. Morrison.
\newblock Asymptotic analysis of a random walk on a hypercube with many dimensions.
\newblock {\em Random Structures \& Algorithms}, 1(1):51--72, 1990.

\bibitem{DS}
Persi Diaconis and Mehrdad Shahshahani.
\newblock Generating a random permutation with random transpositions.
\newblock {\em Z. Wahrsch. Verw. Gebiete}, 57(2):159--179, 1981.

\bibitem{TT2R}
Subhajit Ghosh.
\newblock Total variation cutoff for the transpose top-2 with random shuffle.
\newblock {\em J. Theoret. Probab.}, 33(4):1832--1854, 2020.

\bibitem{hermon2021cutoff}
Jonathan Hermon and Sam Olesker-Taylor.
\newblock Cutoff for almost all random walks on abelian groups.
\newblock {\em arXiv preprint arXiv:2102.02809}, 2021.

\bibitem{Bob_Hough}
Bob Hough.
\newblock The random {$k$} cycle walk on the symmetric group.
\newblock {\em Probab. Theory Related Fields}, 165(1-2):447--482, 2016.

\bibitem{AGG_1}
Xueyi Huang and Qiongxiang Huang.
\newblock The second largest eigenvalues of some {C}ayley graphs on alternating groups.
\newblock {\em J. Algebraic Combin.}, 50(1):99--111, 2019.

\bibitem{AGG_2}
Yanze Huang, Limei Lin, and Dajin Wang.
\newblock On the reliability of alternating group graph-based networks.
\newblock {\em Theoret. Comput. Sci.}, 728:9--28, 2018.

\bibitem{AGG_0}
Yanze Huang, Limei Lin, Li~Xu, and Xiaoding Wang.
\newblock Extra diagnosability and good-neighbor diagnosability of {$n$}-dimensional alternating group graph {$AG_n$} under the {PMC} model.
\newblock {\em Theoret. Comput. Sci.}, 795:36--49, 2019.

\bibitem{jain2024hitting}
Vishesh Jain and Mehtaab Sawhney.
\newblock Hitting time mixing for the random transposition walk.
\newblock {\em arXiv preprint arXiv:2410.23944}, 2024.

\bibitem{JK}
Gordon James and Adalbert Kerber.
\newblock {\em The representation theory of the symmetric group}, volume~16 of {\em Encyclopedia of Mathematics and its Applications}.
\newblock Addison-Wesley Publishing Co., Reading, Mass., 1981.
\newblock With a foreword by P. M. Cohn, With an introduction by Gilbert de B. Robinson.

\bibitem{TT2R_Cayley_graph0}
Jung~Sing Jwo, S.~Lakshmivarahan, and S.~K. Dhall.
\newblock A new class of interconnection networks based on the alternating group.
\newblock {\em Networks}, 23(4):315--326, 1993.

\bibitem{SEP-limit}
Hubert Lacoin.
\newblock The cutoff profile for the simple exclusion process on the circle.
\newblock {\em Ann. Probab.}, 44(5):3399--3430, 2016.

\bibitem{CG-appl}
S~Lakshmivaraha and Sudarshan~K Dhall.
\newblock {\em Analysis and design of parallel algorithms: Arithmetic and matrix problems}.
\newblock McGraw-Hill, Inc., 1990.

\bibitem{LPW}
David~A. Levin, Yuval Peres, and Elizabeth~L. Wilmer.
\newblock {\em Markov chains and mixing times}.
\newblock American Mathematical Society, Providence, RI, 2009.
\newblock With a chapter by James G. Propp and David B. Wilson.

\bibitem{Ramanujan_graph-limit}
Eyal Lubetzky and Yuval Peres.
\newblock Cutoff on all {R}amanujan graphs.
\newblock {\em Geom. Funct. Anal.}, 26(4):1190--1216, 2016.

\bibitem{N-star_limit}
Evita Nestoridi.
\newblock Comparing limit profiles of reversible {M}arkov chains.
\newblock {\em Electron. J. Probab.}, 29:Paper No. 58, 14, 2024.

\bibitem{N-limit}
Evita Nestoridi and Sam Olesker-Taylor.
\newblock Limit profiles for reversible {M}arkov chains.
\newblock {\em Probab. Theory Related Fields}, 182(1-2):157--188, 2022.

\bibitem{Neato-Oles_Tay}
Evita Nestoridi and Sam Olesker-Taylor.
\newblock Limit profiles for projections of random walks on groups.
\newblock {\em Electron. J. Probab.}, 29:Paper No. 158, 22, 2024.

\bibitem{olesker2024limit}
Sam Olesker-Taylor and Dominik Schmid.
\newblock Limit profile for the {B}ernoulli--{L}aplace urn.
\newblock {\em arXiv preprint arXiv:2409.07900}, 2024.

\bibitem{Comp-mix1}
Robin Pemantle.
\newblock A shuffle that mixes sets of any fixed size much faster than it mixes the whole deck.
\newblock {\em Random Structures Algorithms}, 5(5):609--626, 1994.

\bibitem{Ruff}
Oliver Ruff.
\newblock Weight theory for alternating groups.
\newblock {\em Algebra Colloq.}, 15(3):391--404, 2008.

\bibitem{S}
Laurent Saloff-Coste.
\newblock Random walks on finite groups.
\newblock In {\em Probability on discrete structures}, volume 110 of {\em Encyclopaedia Math. Sci.}, pages 263--346. Springer, Berlin, 2004.

\bibitem{Comp-mix2}
Oded Schramm.
\newblock Compositions of random transpositions.
\newblock {\em Israel J. Math.}, 147:221--243, 2005.

\bibitem{Serre}
Jean-Pierre Serre.
\newblock {\em Linear representations of finite groups}.
\newblock Springer-Verlag, New York-Heidelberg, 1977.
\newblock Translated from the second French edition by Leonard L. Scott, Graduate Texts in Mathematics, Vol. 42.

\bibitem{T-limit}
Lucas Teyssier.
\newblock Limit profile for random transpositions.
\newblock {\em Ann. Probab.}, 48(5):2323--2343, 2020.

\bibitem{AGG_3}
Lantao You, Jianxi Fan, Yuejuan Han, and Xiaohua Jia.
\newblock One-to-one disjoint path covers on alternating group graphs.
\newblock {\em Theoret. Comput. Sci.}, 562:146--164, 2015.

\bibitem{AGG_5}
Jin-Xin Zhou.
\newblock The automorphism group of the alternating group graph.
\newblock {\em Appl. Math. Lett.}, 24(2):229--231, 2011.

\end{thebibliography}
	\bibliographystyle{plain}%{alpha}
\end{document}